\documentclass[reqno]{amsart}
\usepackage{amssymb}
\usepackage{amsmath}
\usepackage{amsthm}
\usepackage{mathtools}
\usepackage[svgnames]{xcolor}
\usepackage[pdftex]{hyperref}
\usepackage{enumitem}
\usepackage{booktabs}
\usepackage{placeins }
\usepackage{soul}
\usepackage{enumitem}
\usepackage{verbatim}

\usepackage{tikz}

\definecolor{color1}{RGB}{27,158,119}
\definecolor{color2}{RGB}{217,95,2}
\definecolor{color3}{RGB}{117,112,179}
\definecolor{color4}{RGB}{231,41,138}

\graphicspath{{./figures/}}

\usepackage{vmargin}
\setmarginsrb {2cm}{2cm}{2cm}{2cm} {1cm}{1cm}{0.5cm}{0.5cm}

\hypersetup{
  pdftitle={},
  pdfauthor={Perla Kfoury, Stefan Le Coz},
  pdfsubject={},
  pdfkeywords={}, 
  colorlinks=true,
  linkcolor=DarkBlue , 
  citecolor=DarkRed, 
  filecolor=DarkMagenta, 
  urlcolor=DarkGreen, 
}

\newtheorem{theorem}{Theorem}[section]
\newtheorem{proposition}[theorem]{Proposition}
\newtheorem{lemma}[theorem]{Lemma}

\theoremstyle{remark}
\newtheorem{remark}[theorem]{Remark}

\newcommand{\dual}[2]{\left\langle #1,#2 \right\rangle}

\newcommand{\N}{\mathbb{N}}
\newcommand{\R}{\mathbb{R}}
\newcommand{\C}{\mathbb{C}}
\newcommand{\Z}{\mathbb{Z}}

\renewcommand{\leq}{\leqslant}
\renewcommand{\geq}{\geqslant}

\DeclareMathAlphabet{\mathpzc}{OT1}{pzc}{m}{it}
\renewcommand{\Re}{\mathcal R\!\mathpzc{e}}
\renewcommand{\Im}{\mathcal I\!\mathpzc{m}}

\usepackage{color}

\usepackage{placeins } 

\begin{document}

\title[Variational properties of periodic NLS standing waves]{Variational properties of space-periodic standing waves
  of nonlinear Schr\"odinger equations with general nonlinearities}

\author[P.~Kfoury]{Perla Kfoury}

\address[Perla Kfoury ]{Institut de Math\'ematiques de Toulouse ; UMR5219,
  \newline\indent
  Universit\'e de Toulouse ; CNRS,
  \newline\indent
  INSA IMT, F-31077 Toulouse, 
  \newline\indent
  France}
\email[Perla Kfoury]{perla.kfoury@math.univ-toulouse.fr}

\author[S.~Le Coz]{Stefan Le Coz}
\thanks{The work of P. K. and S. L. C. is 
  partially supported by ANR-11-LABX-0040-CIMI and the ANR project NQG ANR-23-CE40-0005}

\address[Stefan Le Coz]{Institut de Math\'ematiques de Toulouse ; UMR5219,
  \newline\indent
  Universit\'e de Toulouse ; CNRS,
  \newline\indent
  UPS IMT, F-31062 Toulouse Cedex 9,
  \newline\indent
  France}

\email[Stefan Le Coz]{slecoz@math.univ-toulouse.fr}


\date{\today}

\begin{abstract}
  Periodic waves are standing wave solutions of nonlinear Schr\"odinger equations whose profile is periodic in space dimension one. We consider general nonlinearities and provide variational characterizations for the periodic wave profiles. This involves minimizing energy while keeping mass and momentum constant, as well as minimizing the action over the Nehari manifold. These variational approaches are considered both in the periodic and anti-periodic settings, and for focusing and defocusing nonlinearities. In appendix, we study the existence properties of periodic solutions of the triple power nonlinearity.
\end{abstract}


\maketitle

\section{Introduction}
We consider in one space dimension the nonlinear Schr\"odinger equation
\begin{equation}\label{eq:nls4}
  i \psi _t + \psi _{xx} + b f(\psi) =0,
\end{equation}
where $ \psi : \R_t \times \R_x \to \C $, $f\geq 0$ is a gauge invariant nonlinearity and $b \in \R \setminus \{0\}$. Typical examples of nonlinearities that will be considered in this paper are power-type nonlinearities such as $f(z)=|z|^{p-1}z$ for $p>1$, or combinations of powers such as $f(z)=|z|^{p-1}z+|z|^{q-1}z$ for $p,q>1$.

In the present paper, we are interested in specific solutions of \eqref{eq:nls4}, which we call periodic waves. \emph{Periodic waves} are solutions of the type $\psi(t,x)=e^{-iat} u(x)$, for a given frequency $a\in\R$ and a fixed profile $u$ periodic or anti-periodic in space, which verifies an ordinary differential equation (see~\eqref{eq:ode4}). They are the analogues in the context of spatially periodic solutions of the so-called \emph{standing waves} or \emph{solitary waves}, which are solutions of~\eqref{eq:nls4} of the same form, but for which the profile is asked to to be spatially localized, e.g. to live in $H^1(\R)$.

The study of localized solitary waves has a long history, starting with the early works of Strauss~\cite{St77} and Berestycki and Lions~\cite{BeLi83-1} for the existence of solitary waves in higher dimensions, together with the seminal papers of Cazenave and Lions~\cite{CaLi82} and Berestycki and Cazenave~\cite{BeCa81} for the stability and instability by blow-up of solitary waves by variational techniques. The study is still on-going, see e.g. the recent works of Kfoury, Le Coz and Tsai \cite{KfLeTs22} on the stability of standing waves of the nonlinear Schr\"odinger equation with double power nonlinearity and the works of Liu, Tsai and Zwiers~\cite{LiTsZw21} and Morrison and Tsai~\cite{MoTs23} on standing waves of the nonlinear Schr\"odinger equation with triple power nonlinearity.

While there now exists an extensive literature on the existence and stability of solitary waves, the study of periodic waves remains in its infancy, with studies focused mostly on specific nonlinearities such as the cubic power. The aim of the present paper is to extend the existing results beyond the example of the cubic nonlinearity and to treat more generic nonlinearities such as generic powers or sums of generic powers.

Before presenting our main results, we shortly review some of the existing results on periodic waves of nonlinear Schr\"odinger equations. Most of the literature focuses on the cubic case, which is completely integrable and for which explicit expressions of the periodic wave solutions in terms of Jacobi elliptic functions are available. One of the early studies was performed by Rowland in~\cite{Ro74}, where he obtained formally the stability of snoidal wave solutions. Gallay and Haragus~\cite{GaHa07-2,GaHa07-1} proved stability of snoidal waves with respect to same period perturbations using ordinary differential equations analysis and spectral arguments. Bottman, Deconinck and Nivala~\cite{BoDeNi11}, and Deconinck and Segal~\cite{DeSe17} used the complete integrability to obtain an analytical expression for the spectrum (in $L^2(\mathbb R)$) of the linearization of the cubic (focusing and defocusing) nonlinear Schr\"odinger equation around a periodic traveling wave. Orbital stability with respect to any subharmonic perturbation was obtained by Deconinck and Upsal~\cite{DeUp20} and Gallay and Pelinovsky~\cite{GaPe15} via a Lyapunov functional using higher-order conserved quantities of the cubic nonlinear Schr\"odinger equation. Rogue periodic waves, i.e. rogue waves on a periodic background, have been constructed by Chen and Pelinovsky~\cite{ChPe17}. Gustafson, Le Coz and Tsai~\cite{GuLeTs17} provided a global variational characterization of the cnoidal, snoidal, and dnoidal elliptic functions for the cubic case, and proved orbital stability results for the corresponding solutions.

The existence and orbital stability of periodic waves of the cubic-quintic nonlinear Schr\"odinger equation were studied by Alves and Natali in~\cite{AlNa23} using the construction of a smooth curve of dnoidal profiles by bifurcation. Existence and orbital instability results of cnoidal periodic waves for the quintic Klein-Gordon and nonlinear Schr\"odinger equations were obtained by Moraes and de Loreno~\cite{Mode22} using spectral analysis and Shatah-Strauss approach. The periodic traveling wave solutions of the derivative nonlinear Schr\"odinger equation (which is connected to the cubic-quintic nonlinear Schr\"odinger equation) were studied by Hayashi~\cite{Ha19}. A rigorous modulational stability theory for periodic traveling wave solutions to equations of nonlinear Schr\"odinger type with generic nonlinearities was presented by Leisman, Bronski, Johnson and Marangell in~\cite{PlBrJoMa21}, with application to the cubic and quintic nonlinearities.

Most of the works devoted to the study of periodic waves uses tools such as ordinary differential equations analysis, bifurcation theory or spectral theory. In the present paper, we focus on the variational properties of periodic waves, i.e. we characterize them as solutions of minimization problems. The variational problems that we consider are of two types: first, minimization of the energy over the mass constraint; second, minimization of the action over the Nehari manifold. The first type of minimization problems has the advantage of leading at the same time to the orbital stability of the wave obtained. On the other hand, it is of course unable to capture unstable periodic waves. The second type of minimization problems has the advantage of capturing a wider range of periodic waves. On the other hand, obtaining stability or instability requires further investigation.

Under mild assumptions on the nonlinearity, the Cauchy problem for~\eqref{eq:nls4} is locally well-posed (see~\cite{Ca03}) in
the space of periodic functions $H^1_{loc}(\R) \cap P_T$ (as well as in the space of anti-periodic functions $H^1_{loc}(\R) \cap A_T$), i.e. for any $\psi_0\in H^1_{loc}(\R) \cap P_T$ there exists a unique maximal solution $\psi\in C((-T_*,T^*), H^1_{loc}(\R) \cap P_T)$. Moreover, we have continuous dependance with respect to the initial data, the blow-up alternative holds, and the energy, mass and momentum of the solution, defined as follows, are conserved along the time evolution:
\begin{gather*}
  \mathcal{E}(\psi)= \frac{1}{2} \int_0^T | \psi_x|^2 dx- b\int_0^T F(\psi) dx,\\
  M(\psi )= \frac{1}{2}\int_0^T | \psi|^2 dx,\quad \quad P(\psi)= \frac{1}{2} \Im \int_0^T \psi \overline{\psi}_x dx,
\end{gather*}
where by $F$ we denote the real antiderivative of $f$, i.e. $F(z) = \int_0^{|z|}f(s) ds$ and $T$ is the space period.

It is common to consider the so-called \emph{action} functional, defined for a given $a$ by
\[
  S(\psi)=\mathcal{E}(\psi)-a M(\psi),
\]
and the associated \emph{Nehari} functional, defined by
\[
  I(\psi)=\dual{S'(\psi)}{\psi}= \int_0^T | \psi_x|^2 dx -a\int_0^T | \psi|^2 dx- b\int_0^T\Re( f(\psi) \overline{\psi}) dx
  .
\]

Periodic waves can be obtained as solutions of various variational problems. As already said, in the present paper, we consider two variational problems in particular : minimization of the energy over fixed mass (and sometimes momentum) and minimization of the action over the Nehari manifold. We have investigated these minimization problems for periodic and anti-periodic functions. Our results can be summarized informally as follows (see Section~\ref{sec:minimization} for precise statements).

\begin{theorem}\label{thm} Let the energy, mass, momentum, action and Nehari functionals be defined as above. The following holds.
  \begin{enumerate}
  \item Let $b>0$. There exists a real valued minimizer of the energy, under fixed mass or under fixed mass and zero momentum, among periodic functions. The minimal energy is finite and negative. If the mass is larger than a given threshold, then the minimizer is not a constant, the associated Lagrange multiplier verifies $a<0$, and the minimizer is positive.    

  \item Let $b<0$. There exists a unique (up to phase shift) minimizer of the energy, under fixed mass or under fixed mass and zero momentum, among periodic functions. It is the constant function $u_{\infty} \equiv \sqrt{\frac{2m}{T}}.$
    
  \item Let $b<0$ and $f(u)=\sum_{j=1}^N|u|^{p_j-1}u$,
    with $p_j>1$ for $j=1,\dots,N$. There exists a unique (up to phase shift and complex conjugate) minimizer of the energy, under fixed mass, among anti-periodic functions. It is the plane wave $u_{\infty} \equiv \sqrt{\frac{2m}{T}} e^{\frac{i \pi x}{T}}$.
  \item Let $b>0$, $a<0$ and $f(u)= |u|^{p-1}u$, with
    $p>1$. The minimum of the action on
    the Nehari manifold among periodic functions is
    finite and there exists a real minimizer.

    \item Let $b<0$, $a>0$ and $f(u)= |u|^{p-1}u$, with
    $p>1$. 
    There exists a unique (up to phase shift) minimizer of  the action on
    the Nehari manifold among periodic functions. It is the constant function $u_{\infty} \equiv\left(\frac{-a}{b}\right)^{\frac{1}{p-1}}.$

  \item Let $b>0$, $a<\frac{4\pi^2}{T^2}$ and $f(u)= |u|^{p-1}u$, with $p>1$. The minimum of the minimization problem on the Nehari manifold among anti-periodic functions is finite and there exists a minimizer.
    When $p$ is an odd integer then the minimizer is real.
  \end{enumerate}
\end{theorem}

To establish the results of Theorem~\ref{thm}, we proceed in the following way. We start by an analysis of the ordinary differential equation verified by the profile. We then consider the variational problems themselves. As we are working with periodic functions, hence on the bounded domain $[0,T]$, we benefit from the compactness properties of Sobolev injections and existence of a minimizer is straightforward in most cases. On the other hand, the identification of the minimizer, or its specific properties, are usually delicate to obtain. In particular, in the case of minimization over the Nehari manifold for anti-periodic function (see Section~\ref{sec:nehari-anti}), the fact that the minimizer is real valued is established thanks to a Fourier rearrangement inequality which we believe to be new and of independent interest (see Lemma~\ref{lem:new}).

The rest of the paper is organized as follows. In Section~\ref{sec:profile}, we present the notation, the assumptions on the nonlinearity and the analysis of the profile ordinary differential equation. In Section~\ref{sec:minimization}, we present a Fourier rearrangement inequality, then we study minimization of the energy at fixed mass, and finally we consider minimization over the Nehari manifold. The appendix gathers related material which was not fitting directly into the study: we study the existence properties of periodic solutions to the equation with triple power nonlinearity.



\section{Analysis of the profile equation}
\label{sec:profile}

The simplest non-trivial solutions of~\eqref{eq:nls4} are the \emph{standing waves}. They are solutions of the form
\begin{equation*}
  \psi (t,x)= e^{-iat} u(x), \quad a \in \R.
\end{equation*}
The profile function $u(x)$ satisfies the ordinary differential equation
\begin{equation} \label{eq:ode4}
  u_{xx}+au+b f(u)=0.
\end{equation}
It is an integrable ordinary differential equation, whose conserved quantities (on $\C$) are the momentum $J$ and the energy $E$, given by
\[
  J= \Im (u_x \overline{u}), \quad \quad E= \frac{1}{2} |u_x |^2 + \frac{a}{2} |u|^2 + b F(u).
\] 
The nonlinearity $f: \C \to \C $ is defined for any $z \in \C $ by $
f(z)=g(|z|^2 )z $ with $ g \in C^0 ( [0, + \infty), \R) \cap C^2 ((0,
+ \infty), \R) $. For simplicity, we denote by $f'$ the derivative
of $f_{|\R}$. For the analysis of the ordinary differential equation in the present section, we make the following assumptions on the
nonlinearity.
\begin{enumerate}[start=1,label={(H\arabic*)}]
\item \label{item:h1}
  The derivative $f'$ is
  non-decreasing on $(0,\infty)$.
\item \label{item:h2} At infinity $\lim_{s\to\infty}\frac{F(r)}{r^2}=\infty$.
\item \label{item:h3} The function $s\to \frac{f(s)}{s}$ is increasing on
  $(0,\infty)$ and $\lim_{s\to 0}\frac{f(r)}{r}=0$. 
\end{enumerate}

For the rest of this section, we assume that
\ref{item:h1}-\ref{item:h3} hold and we study the bounded solutions of the profile
equation~\eqref{eq:ode4}. We will distinguish between two different
cases depending on whether $J=0$ or not. When $u(x)\neq 0$, we introduce the polar coordinates 
\[
  u(x)= r(x) e^{i \phi (x)} ,
\]
with $r>0$ and $r , \phi \in C^2(\R)$. Invariants become 
\[
  J=r^2 \phi _x, \quad 
  E= \frac{{r_x}^2}{2} + \frac{J^2}{2r^2} + a\frac{r^2}{2}+b F(r).
\]
If $ J=0$, then replacing $r(x)e^{i\phi(x)}$ with $r(x)e^{i\phi}$ for
some $\phi \in [0, 2\pi]$ we can assume that $u(x) \in \R$ up to a
constant phase. If $J \neq 0 $, then $u(x) \neq 0$ for all $x \in \R
$, and $\phi_x \neq 0$, so $u$ is truly complex-valued.
Define the effective potential by
\[
  V_J(r) = \frac{J^2}{2r^2} + a\frac{r^2}{2} +b F(r).
\]
By elementary calculations, we have
\[ 
  V'_J(r)= - \frac{J^2}{r^3}+a r +bf(r).
\]
We describe the potential $V_J$. We start with the case $J=0$. Then 
\[
  V(r)= a \frac{r^2}{2}+ b F(r), \quad V'(r)=a r+b f(r).
\]
If $V'(r)=0$ for $r>0$, then $\frac{f(r)}{r}= -\frac{a}{b}$. We know that $\frac{f(r)}{r}$ is an increasing function for all $r > 0$, therefore there exists at most one value $r_0 >0$ such that 
\[
  a r_0+b f(r_0)=0.
\]
We now discuss what happens depending on the values of $a$ and $b$.
Since $\lim\limits_{r \rightarrow 0} \frac{f(r)}{r}=0$, we have $ sign
(V'(r))= sign (a)$, when $r$ approaches $0^+$. Moreover, we have $
\lim\limits_{r \rightarrow + \infty} V(r)= sign(b) \infty$, because we
have $\lim\limits_{r \rightarrow +\infty}
\frac{F(r)}{r^2}=+\infty$. We start with the defocusing case where
$b<0$. If $a\leq 0$, then $V'(r)<0$ for all $r>0$ and there does not
exist bounded solutions. Assume that $a>0$. Then $V'(r)=0$ has exactly one solution. 
Therefore the graph of $ V $ as a function of $r$ is given on the left of Figure~\ref{fig:V(r)}.
The third case is the focusing case where $b>0$ with $a\geq 0$. Then
$V'(r)>0$ for all $r>0$. The graph of $ V $ as a function of $r$ is represented on the center of Figure~\ref{fig:V(r)}. The last case is the focusing case where $b>0$ with $a<0$, then $V'(r)=0$ has again exactly one solution, and the graph of $ V $ as a function of $r$ is represented on the right of Figure~\ref{fig:V(r)}.
\begin{figure}[htbp!]
  \centering
  \begin{tabular}{cc}
    \includegraphics[width=.31\textwidth]{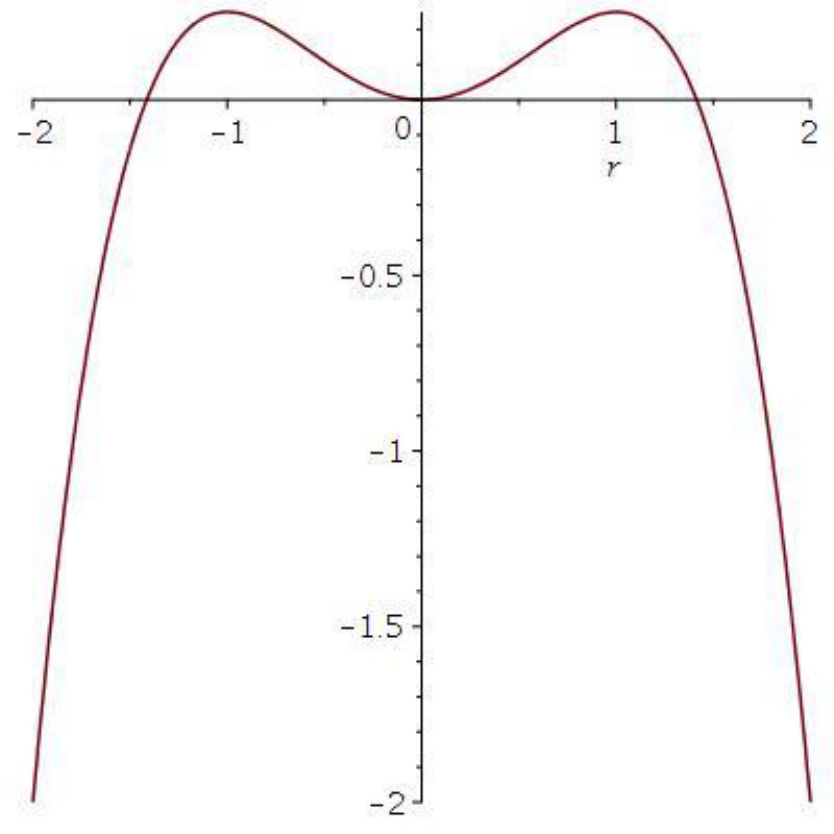} 
    \includegraphics[width=.31\textwidth]{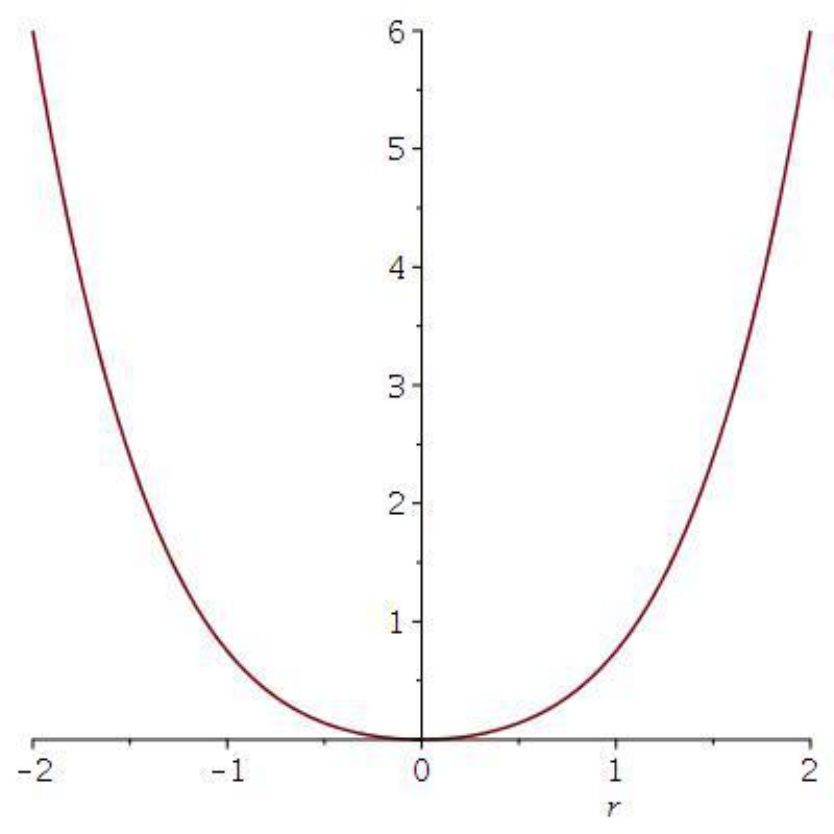}
    \includegraphics[width=.31\textwidth]{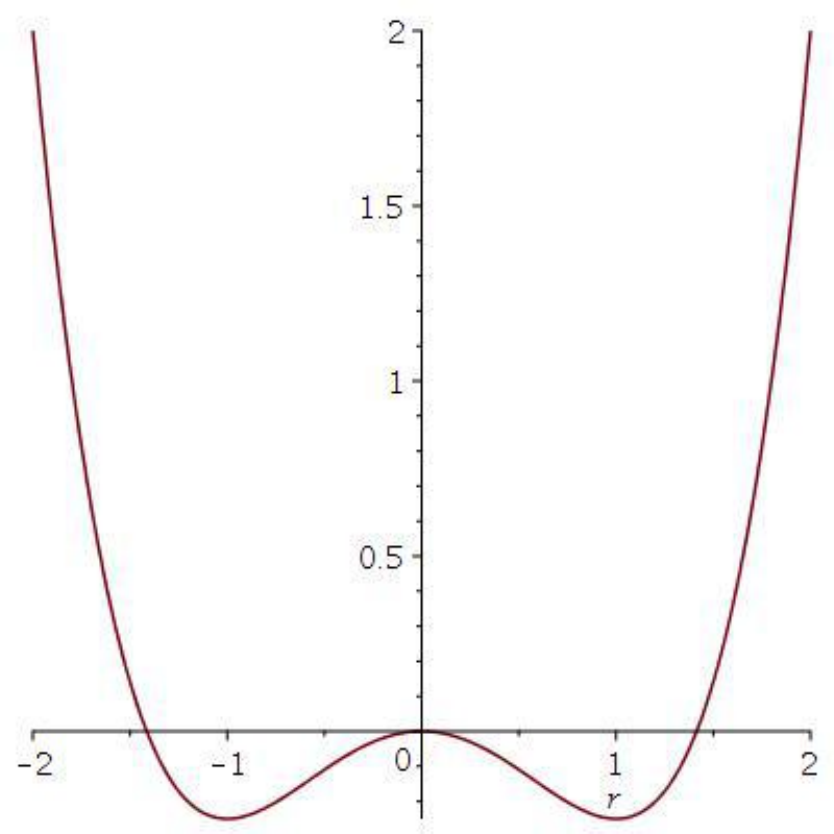}
  \end{tabular}
  \caption{Possible plots of $V$ as a function of $r$ when $J=0$.}
  \label{fig:V(r)}
\end{figure}

Now we assume that $J \neq 0$.
If $V'_J(r)=0 $ for $r>0$, then $-J^2+ r^4\left(a +b \frac{f(r)}{r}\right)=0$.
Let 
\[
  k(r)= r^4\left(a +b \frac{f(r)}{r}\right).
\]
We will study the variations of the function $k$, and infer from these the graph of the potential. We have
\[
  k'(r)= 4 a r^3 +b f'(r) r^3 +3 b r^2 f(r)=r^3 \left( 4 a + b f'(r) + 3 b \frac{f(r)}{r}\right).
\]
As $f'(r)$ and $f(r)/r$ are increasing, the derivative $k'$ changes
sign at most once. When $b<0$ and $a\leq0$,
$k$ is always negative decreasing, $V_J$ has no critical point and
\eqref{eq:ode4} has no bounded solution. 
We then consider the defocusing case where $b<0$ and $a>0$. In this
case the graph of $k(r)$ as a function of $r$ is presented on the left
of Figure~\ref{fig:k(r)}. Hence $V'_J(r)=0$ has 2 solutions for
$J^2<r_c$ where $k'(r_c)=0$ and the maximum occurs. The graph of $V_J$
as a function of $r$ is presented on the left of Figure
\ref{fig:V_J(r)}. In that case, $V_J$ admits a minimum and a
maximum. When $J^2=r_c$, $V$ is monotonically decreasing but still
admits a critical point, which gives rise to a unique bounded solution 
of~\eqref{eq:ode4} (which is a plane wave). When $J^2>r_c$, $V$ is
monotonically decreasing with no critical point, and
\eqref{eq:ode4} has no bounded solution. 
The third case is the focusing case where $b>0$ and $a\geq 0$. We know that $k(r)$ is a strictly increasing function presented on the center of Figure~\ref{fig:k(r)}. Therefore, $V'_J(r)=0$ has a unique solution. Then the graph of $V_J$ as a function of $r$ is given in the center of Figure~\ref{fig:V_J(r)}.
The last case is the focusing case where $b>0$ and $a<0$. The graph of $k$ as a function of $r$ is represented on the right of Figure~\ref{fig:k(r)}. Therefore, $V'_J(r)=0$ has a unique solution. Then the graph of $V_J$ as a function of $r$ is represented on the right of Figure~\ref{fig:V_J(r)}.

\begin{figure}[htbp!]
  \centering
  \begin{tabular}{cc}
    \includegraphics[width=.31\textwidth]{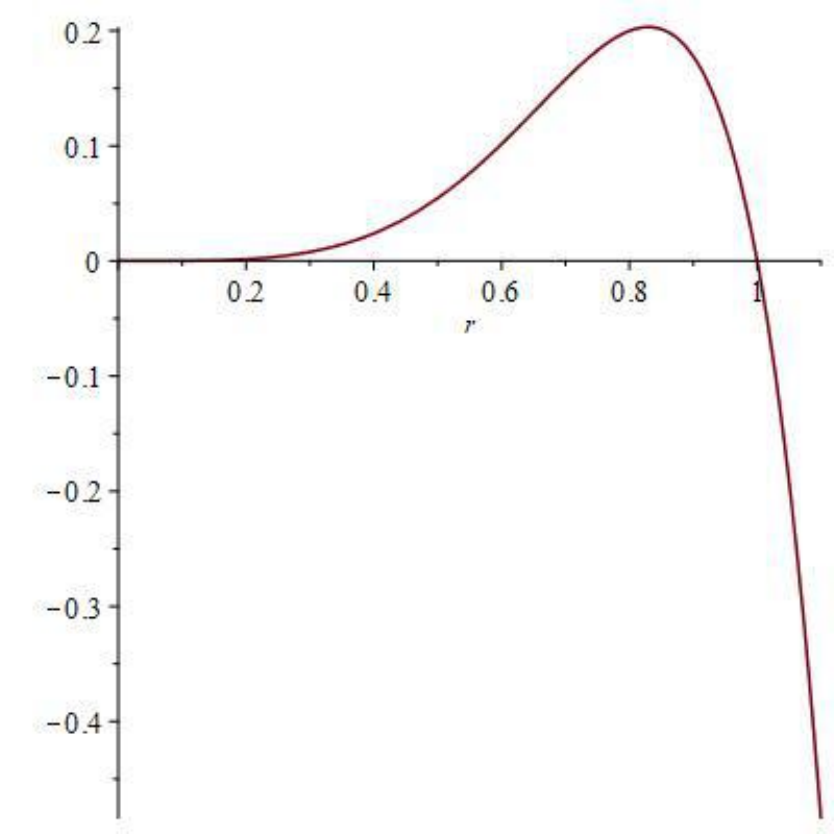}
    \includegraphics[width=.31\textwidth]{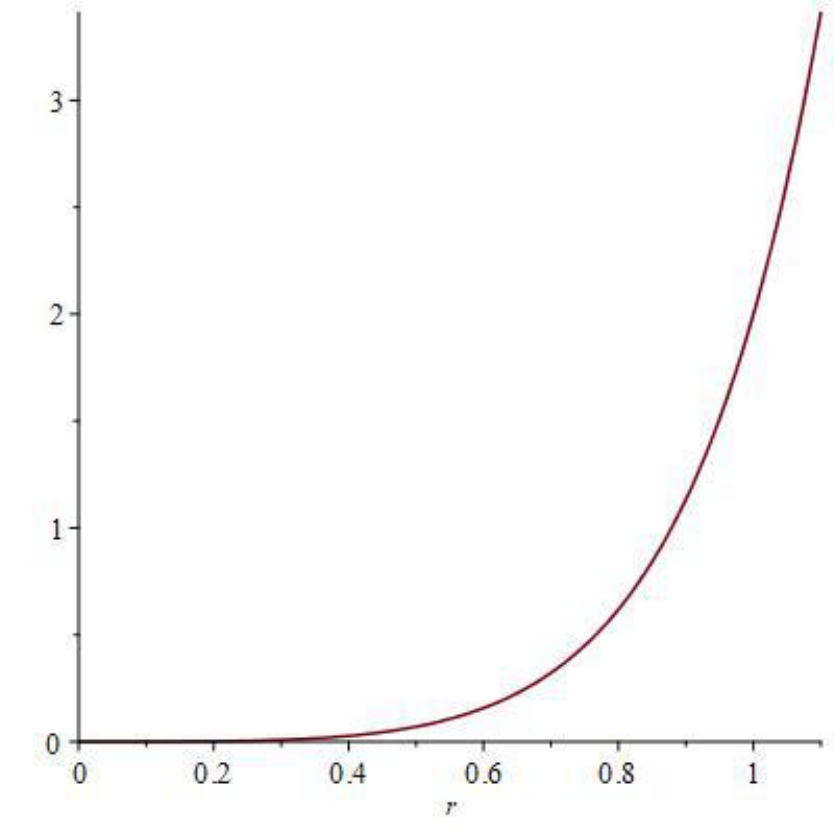}
    \includegraphics[width=.31\textwidth]{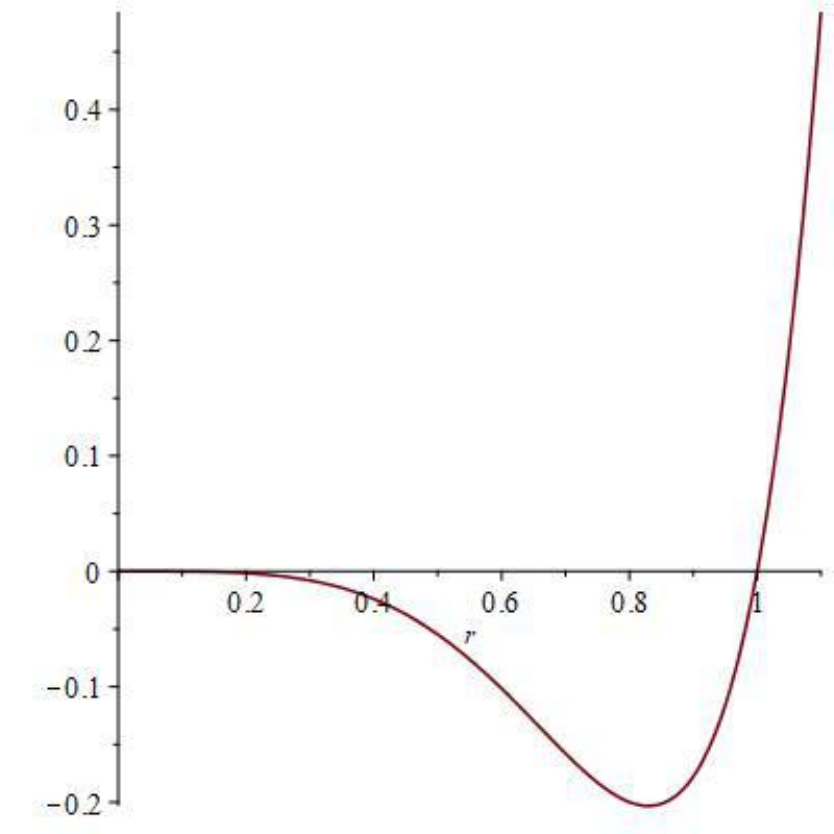}
  \end{tabular}
  \caption{Possible plots of $k$ as a function of $r$.}
  \label{fig:k(r)}
\end{figure}

\begin{figure}[htbp!]
  \centering
  \begin{tabular}{cc}
    \includegraphics[width=.31\textwidth]{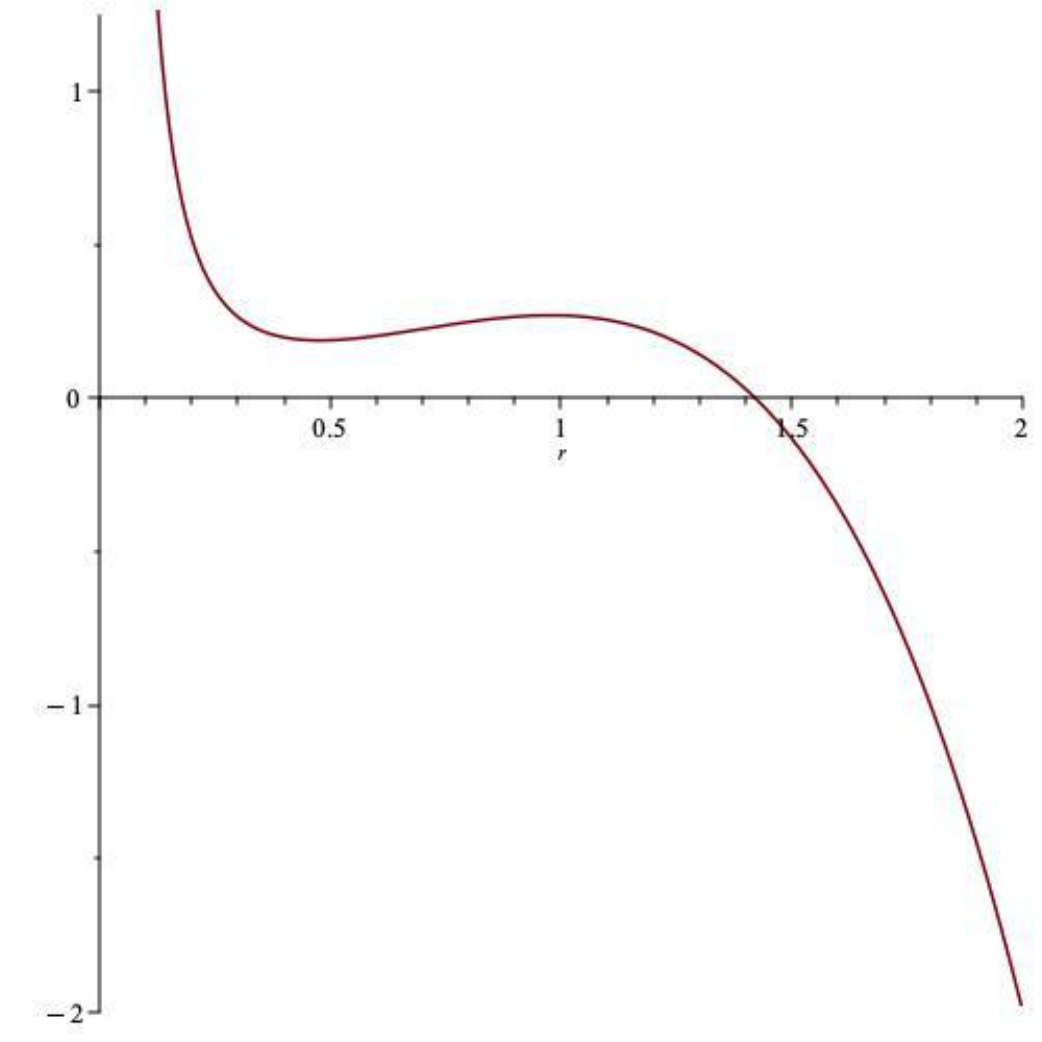} 
    \includegraphics[width=.31\textwidth]{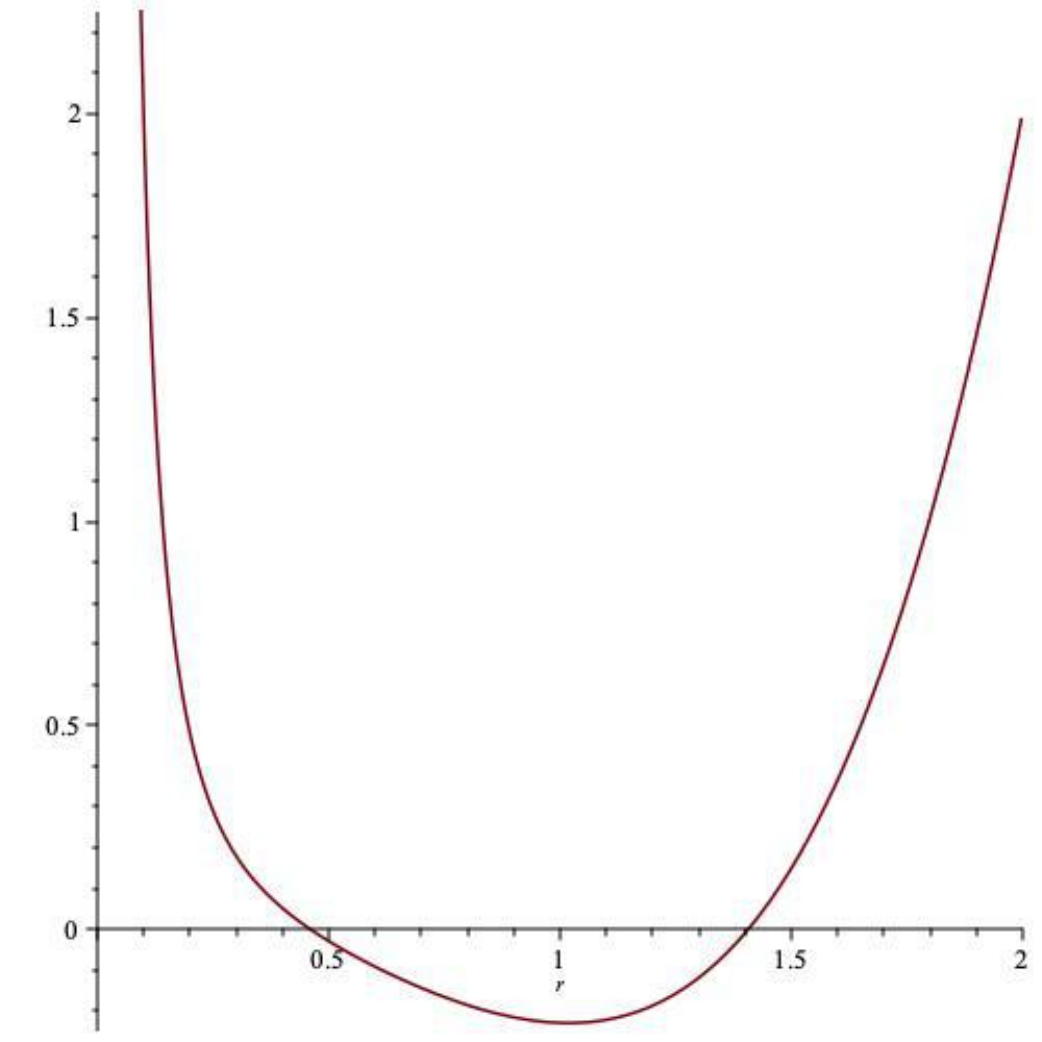}
    \includegraphics[width=.31\textwidth]{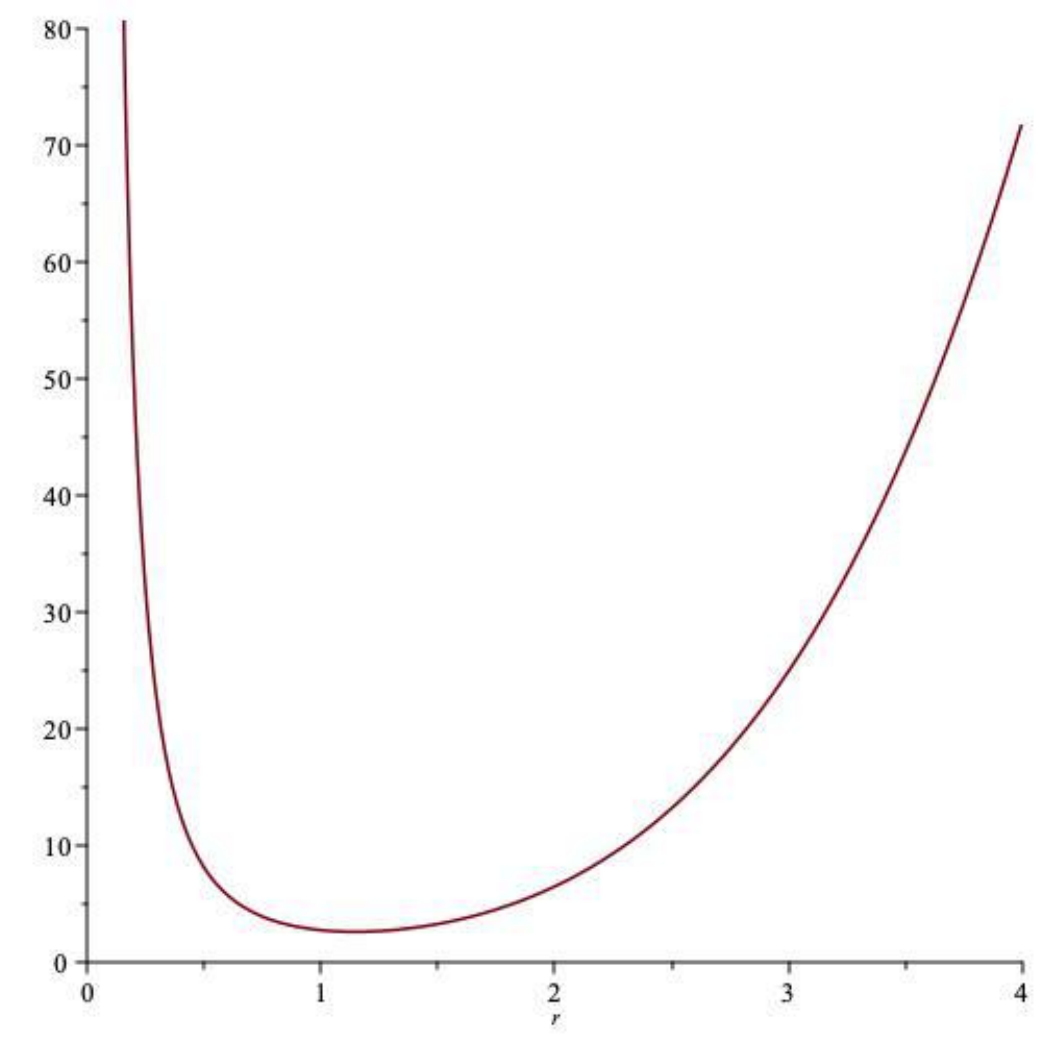}
  \end{tabular}
  \caption{Possible plots of $V_J$ as a function of $r$ with $J \neq 0$.}
  \label{fig:V_J(r)}
\end{figure}

We now represent the phase portraits in each case where bounded
solutions exist. In polar coordinates, the equation~\eqref{eq:ode4} becomes
\[
  \label{eq:unkown}
  r_{xx} - \frac{J^2}{r^3} + a r +b f(r)=0.
\]
We rewrite this second-order differential equation in the form of a first-order system by introducing new coordinates
\[
  y= \begin{pmatrix}
    y_1 \\
    y_2
  \end{pmatrix}= \begin{pmatrix}
    r \\
    r_x
  \end{pmatrix}.
\]
Then the differential system is the following 
\[
  y'= G(y)= \begin{pmatrix}
    y_2 \\
    \frac{J^2}{y_1^3} -a y_1 - bf(y_1)
  \end{pmatrix} = \begin{pmatrix}
    f_1(y_1,y_2) \\
    f_2(y_1,y_2)
  \end{pmatrix}.
\]
We start by finding the equilibrium points $y$ such that $G(y)=0.$ Then we find the isoclines $I_0$ and $I_{\infty}$, where 
\[
  I_{\alpha}= \left\{ (y_1,y_2) \in \R ^2:\frac{f_2 (y_1,y_2)}{f_1(y_1,y_2)}= \alpha \right\} .
\]
We start with the case $J=0$. We have 
\[
  I_0= \{ (y_1,y_2) \in \R^2 :y_2 \neq 0, a y_1 +b f(y_1)=0 \},
\]
and
\[
  I_{\infty}= \{ (y_1,y_2) \in \R^2: -a y_1 -b f(y_1) \neq 0, \quad y_2=0 \}.
\]
These isoclines $I_0$ and $I_{\infty}$ meet at the equilibrium points of the system and determine the regions where the trajectories are monotonic:
\begin{align*}
  Q_{++} &= \{ y \in \R ^2 , \quad f_1(y)>0 , \quad f_2(y)>0 \}.\\
  Q_{+-} &= \{ y \in \R ^2 , \quad f_1(y)>0 , \quad f_2(y)<0 \}.\\
  Q_{-+} &= \{ y \in \R ^2 , \quad f_1(y)<0 , \quad f_2(y)>0 \}.\\
  Q_{--} &= \{ y \in \R ^2 , \quad f_1(y)<0 , \quad f_2(y)<0 \}.
\end{align*}
Then we study the stability of the equilibrium points. The Jacobian matrix of $G$ is of the form 
\[
  J_G= \begin{pmatrix}
    0 & 1\\
    -a -b f'(y_1) & 0
  \end{pmatrix}.
\]
Classification of equilibrium points is determined by the eigenvalues
$\lambda_1$ and $\lambda_2$ of the Jacobian matrix $J_G$. Since the
trace of $J_G$ is $0$, the eigenvalues verify
$\lambda_1=-\lambda_2$. Depending on the discriminant of
$J_G$, two situations may arise. If $\lambda_1=-\lambda_2\neq 0$ (or $\lambda_1=-\lambda_2 =0$ and $b<0$) are real numbers,
then the point is a saddle.
If $\lambda_1=-\lambda_2\neq 0$ are purely imaginary
numbers   (or $\lambda_1=-\lambda_2 =0$ and $b>0$) then the point is a center.

We start with the defocusing case where $b<0$ and $a>0$. We know that $\frac{f(r)}{r}$ is an increasing function on $(0, \infty)$ therefore in this case there exists a unique $r_0>0$ such that $a r_0 +b f(r_0)=0$. Thus we have three equilibrium points $(0,0)$, $(r_0,0)$ and $(-r_0,0)$. 
Hence 
\begin{equation}
  \label{eq:iso_0}
  I_0= \{ (y_1,y_2) \in \R^2: y_1 \in \{0,\pm r_0\},y_2 \neq 0
  \},
\end{equation}
and 
\begin{equation}
  \label{eq:iso_inf}
  I_{\infty}= \{ (y_1,y_2) \in \R^2 : y_1 \notin \{0,\pm r_0\}, y_2=0\}.
\end{equation}
The characteristic polynomial of the Jacobian matrix $J_G$ is given by $P(\lambda)=\lambda^2+a+bf'(y_1)$. At the equilibrium point $(0,0)$ the eigenvalues are $\lambda= \pm i \sqrt{a}$ (recall that $a>0$). Since the eigenvalues are purely imaginary, the equilibrium point $(0,0)$ is a center. At the equilibrium points $(\pm r_0,0)$ we have $a+bf'(r_0)<0$, therefore the eigenvalues are non-zero real numbers of opposite signs and the equilibrium point is a saddle point. The phase portrait is given in Figure~\ref{fig:phase-portrait-def-J=0}.

For the focusing case where $b>0$ with $a>0$ we have only one equilibrium point $(0,0)$. The eigenvalues are given at the equilibrium point $(0,0)$ by $\lambda= \pm i \sqrt{a}$ and the equilibrium point $(0,0)$ is a center. The phase portrait is given on the left of Figure of~\ref{fig:phase-portrait-foc-J=0}. 

The last case is the focusing case with $b>0$ and $a<0$.
There exists a unique $r_0>0$ such that $a r_0 +b f(r_0)=0$ and we have three equilibrium points: $(0,0)$, $(r_0,0)$ and $(-r_0,0)$.
As before, the isoclines $I_0$ and $I_{\infty}$ are given by
\eqref{eq:iso_0} and~\eqref{eq:iso_inf}. At the equilibrium point
$(0,0)$ the eigenvalues are $\lambda= \pm \sqrt{-a}$ and the
equilibrium point $(0,0)$ is a saddle. At the equilibrium points $(\pm
r_0,0)$ the eigenvalues are non zero purely imaginary numbers hence the equilibrium point is a center. The phase portrait is given on the right of Figure~\ref{fig:phase-portrait-foc-J=0}.

\begin{figure}[htbp!]
  \centering
  \begin{tabular}{cc}
    \includegraphics[width=.5\textwidth]{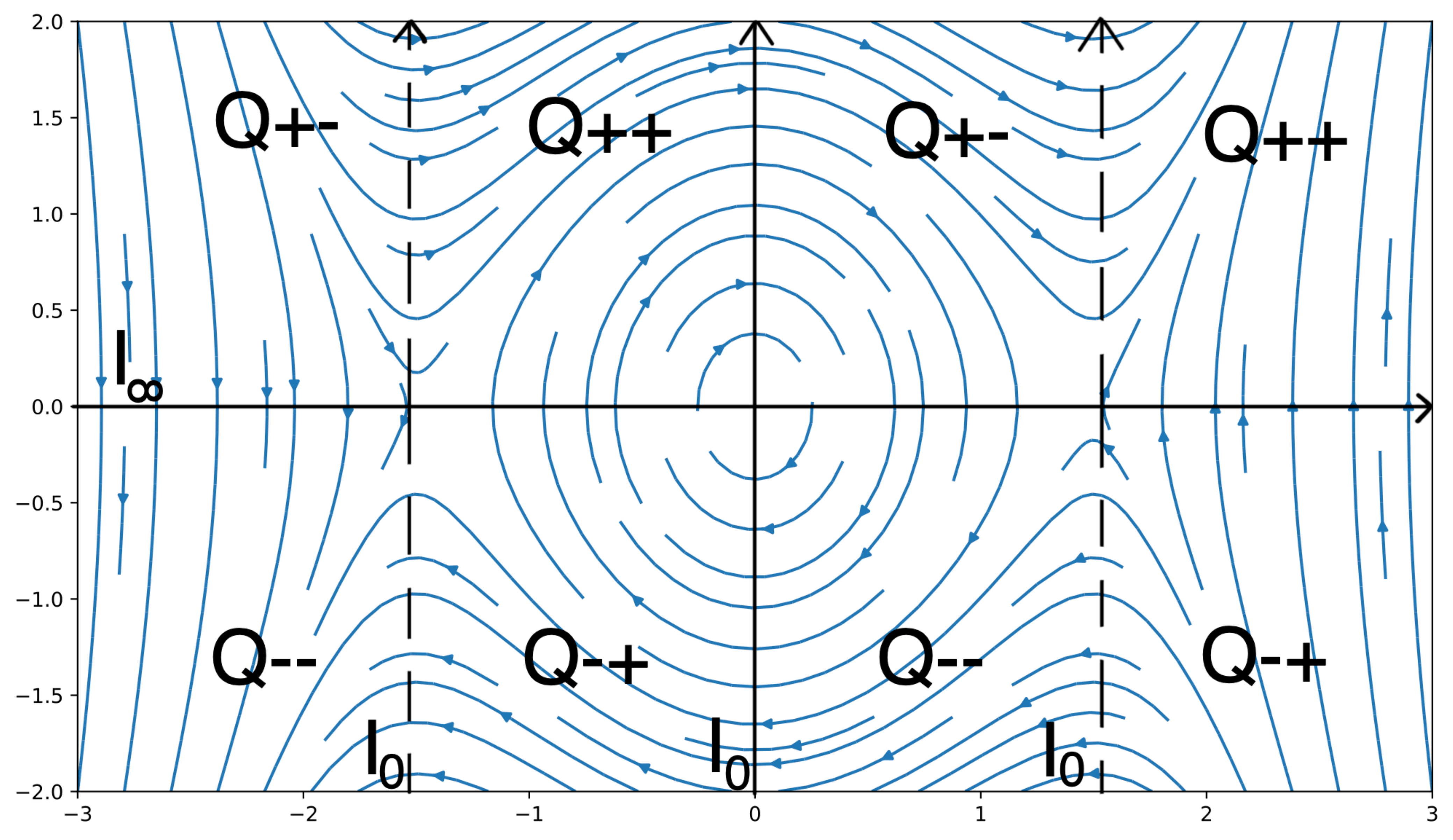} 
  \end{tabular}
  \caption{Examples of phase portraits of the solutions for the defocusing case when $J=0$.}
  \label{fig:phase-portrait-def-J=0}
\end{figure}

\begin{figure}[htbp!]
  \centering
  \begin{tabular}{cc}
    \includegraphics[width=.45\textwidth]{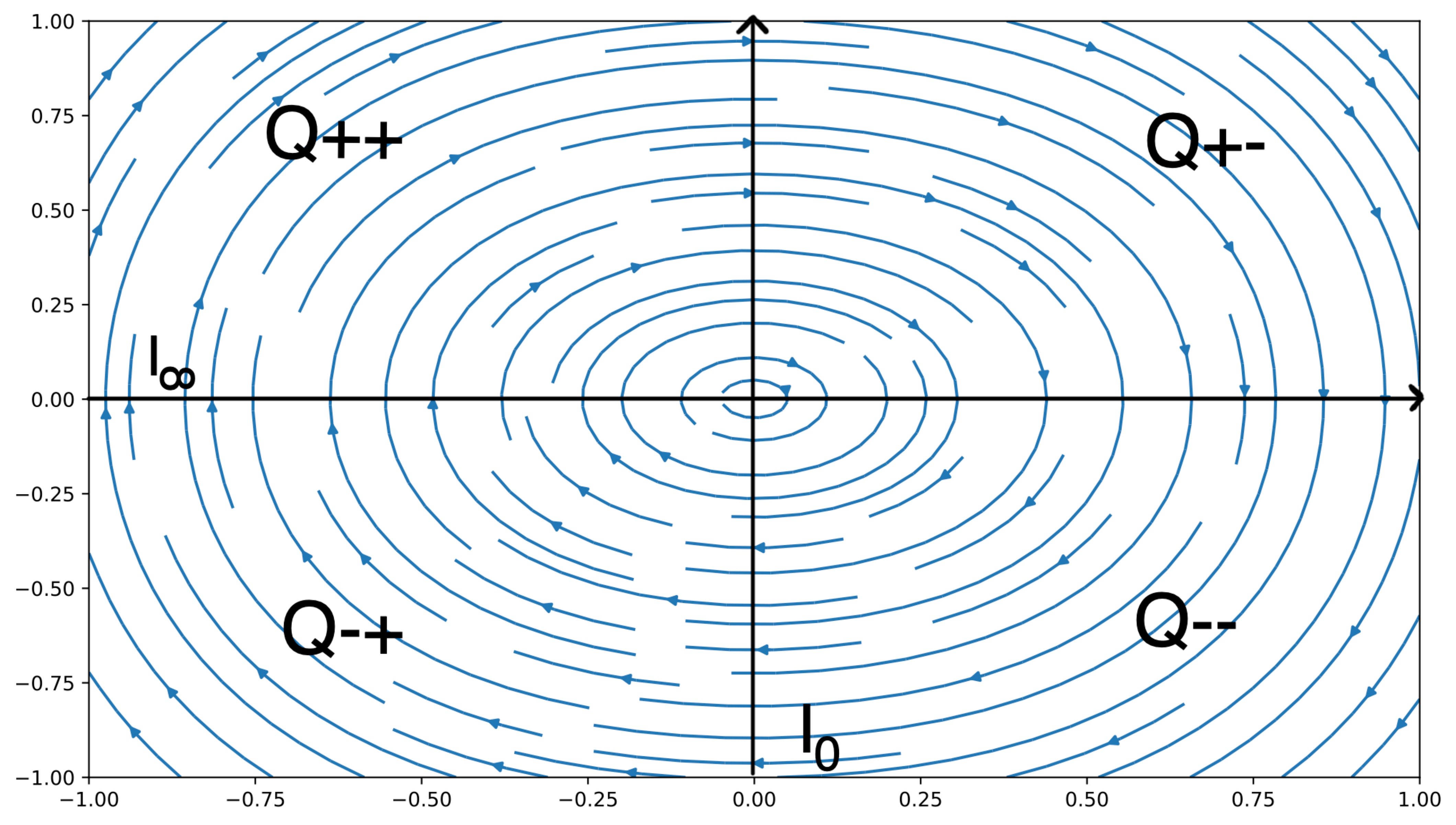}
    \includegraphics[width=.45\textwidth]{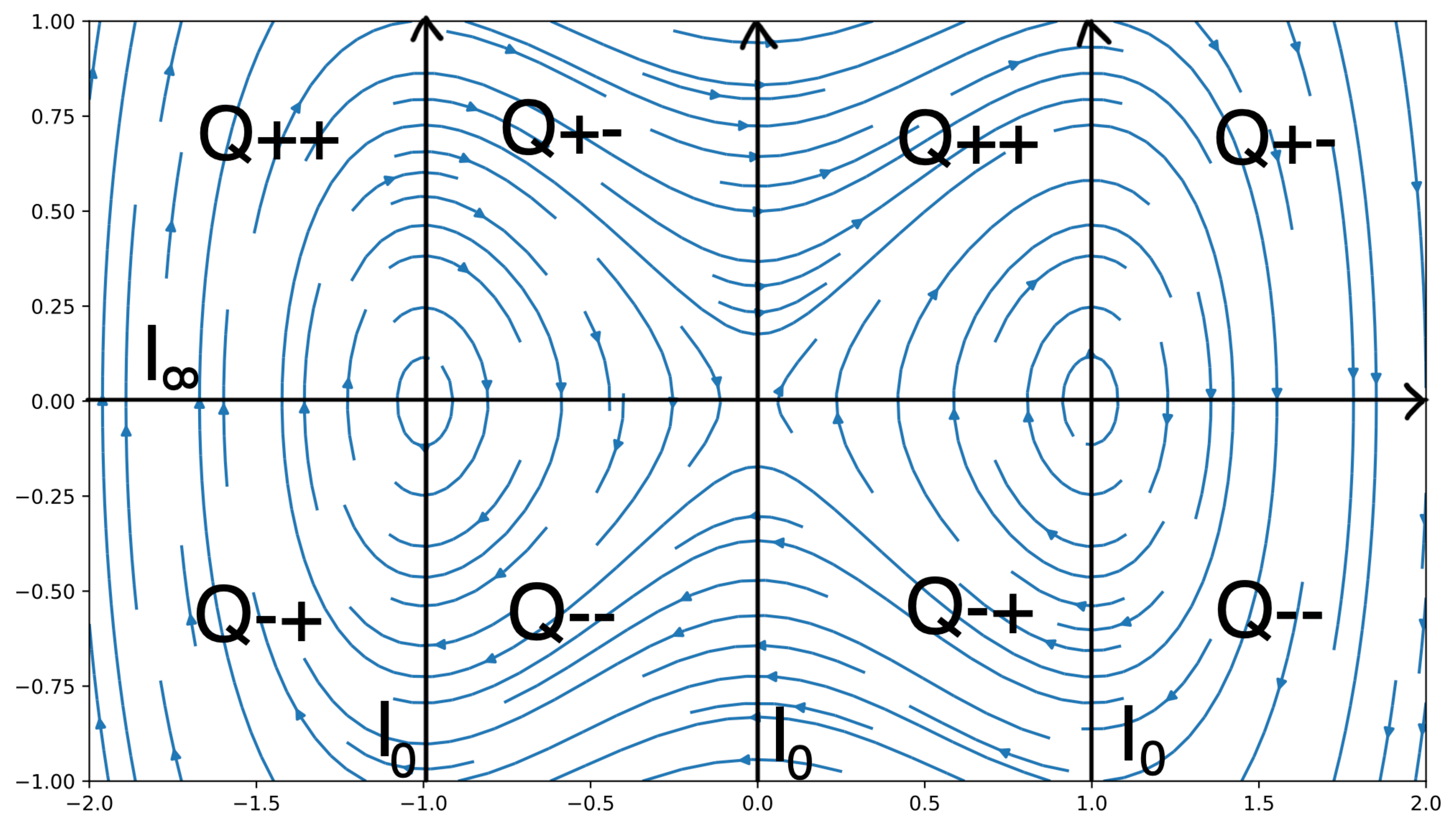}
  \end{tabular}
  \caption{Examples of phase portraits of the solutions for the focusing case when $J=0$.}
  \label{fig:phase-portrait-foc-J=0}
\end{figure}

The second case is when $J \neq 0$. We have 
\[
  I_0= \left\{ (y_1,y_2) \in \R^2: \frac{J^2}{y_1^3} -a y_1 -b f(y_1)=0 , y_2 \neq 0\right\},
\]
and
\[
  I_{\infty} = \{ (y_1,y_2) \in \R^2: \frac{J^2}{y_1^3} -a y_1 -b f(y_1) \neq 0, y_2=0 \}.
\]
The Jacobian matrix of $G$ is of the form 
\[
  J_G= \begin{pmatrix}
    0 & 1\\
    -3 \frac{J^2}{y_1^4}-a -b f'(y_1) & 0
  \end{pmatrix}.
\]

We start with the defocusing case where $b<0$, $a>0$ and $J^2<r_c$. In this case
the equation $\frac{J^2}{y_1^4}-ay_1-bf(y_1)=0$ has 2 solutions $r_Q$
and $r_q$ such that $0<r_Q<r_q$ ($r_Q=r_q$ if $J^2=r_c$). Thus we have
two equilibrium points
$(r_Q,0)$ and $(r_q,0)$. The characteristic polynomial of the Jacobian
matrix $J_G$ is given by $P(\lambda)=\lambda^2+
\frac{3J^2}{y_1^4}+a+bf'(y_1)$. 
On the first equilibrium point $(r_Q,0)$ we have $\lambda^2 =
-\frac{3J^2}{r_Q^4}-1+f'(r_Q)=-V''_J(r_Q)<0$, because $V_J(r)$ is
convex at $r_Q$ and therefore the eigenvalues are purely imaginary and
the equilibrium point $(r_Q,0)$ is a center. On the second equilibrium
point $(r_q,0)$ we have $\lambda^2 =
-\frac{3J^2}{r_q^4}-1+f'(r_q)=-V''_J(r_q)>0$, because $V_J(r)$ is
concave at $r_q$ and therefore the eigenvalues are non-zero real
numbers of opposite signs hence the equilibrium point is a saddle. The
phase portrait is given on the left of Figure
\ref{fig:phase-portrait-J-diff-0}. When $J^2=r_c$, the equilibrium
point is the only bounded solution and is a saddle-node. The phase portrait for this case is given in Figure~\ref{fig:phase-portrait-J-diff-0-crit}.

\begin{figure}[htbp!]
  \centering
  \begin{tabular}{cc}
    \includegraphics[width=.45\textwidth]{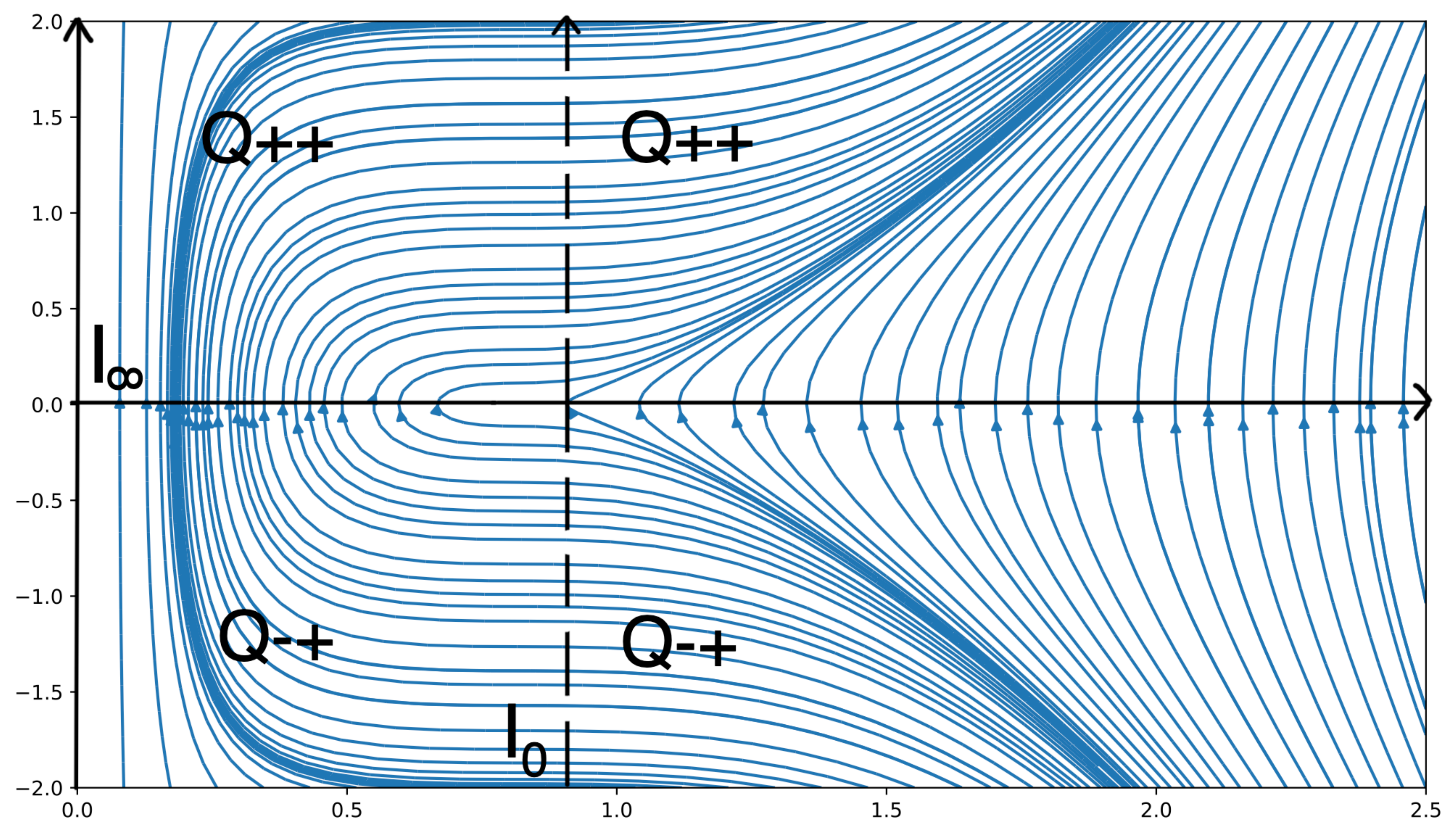}
  \end{tabular}
  \caption{Example of phase portrait of the solutions when $J^2=r_c$.}
  \label{fig:phase-portrait-J-diff-0-crit}
\end{figure}

For the focusing case where $b>0$, with both cases $a>0$ or $a<0$, the equation $\frac{J^2}{y_1^3}-ay_1-bf(y_1)=0$ has 1 solution $r_Q$. On the equilibrium point $(r_Q,0)$ we have $\lambda^2 = -\frac{3J^2}{r_Q^4}-a-f'(r_Q)<0$, therefore the eigenvalues are purely imaginary and the equilibrium point $(r_Q,0)$ is a center. The phase portrait for these two cases is given on the right of Figure~\ref{fig:phase-portrait-J-diff-0}.

\begin{figure}[htbp!]
  \centering
  \begin{tabular}{cc}
    \includegraphics[width=.45\textwidth]{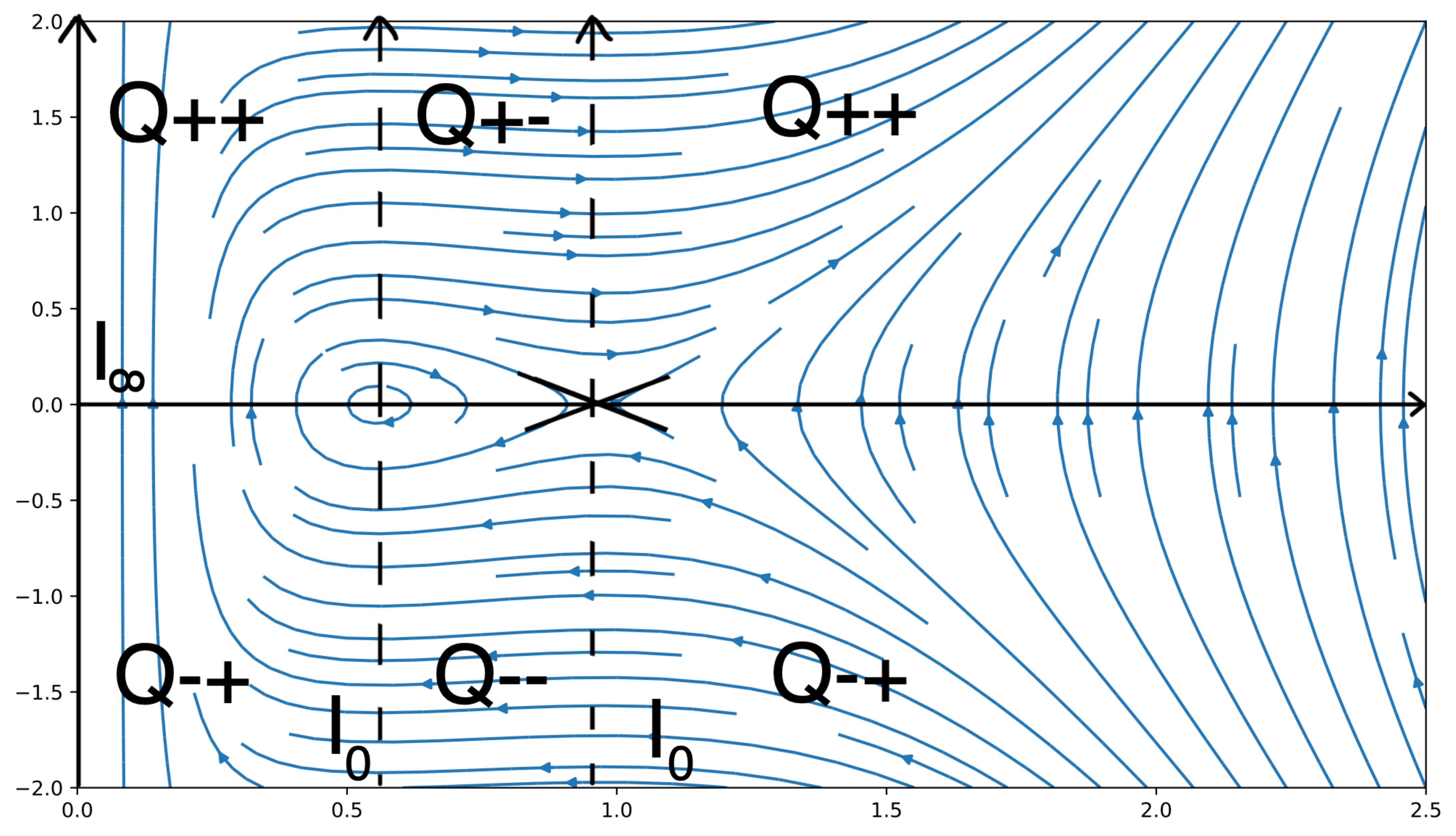}
    \includegraphics[width=.45\textwidth]{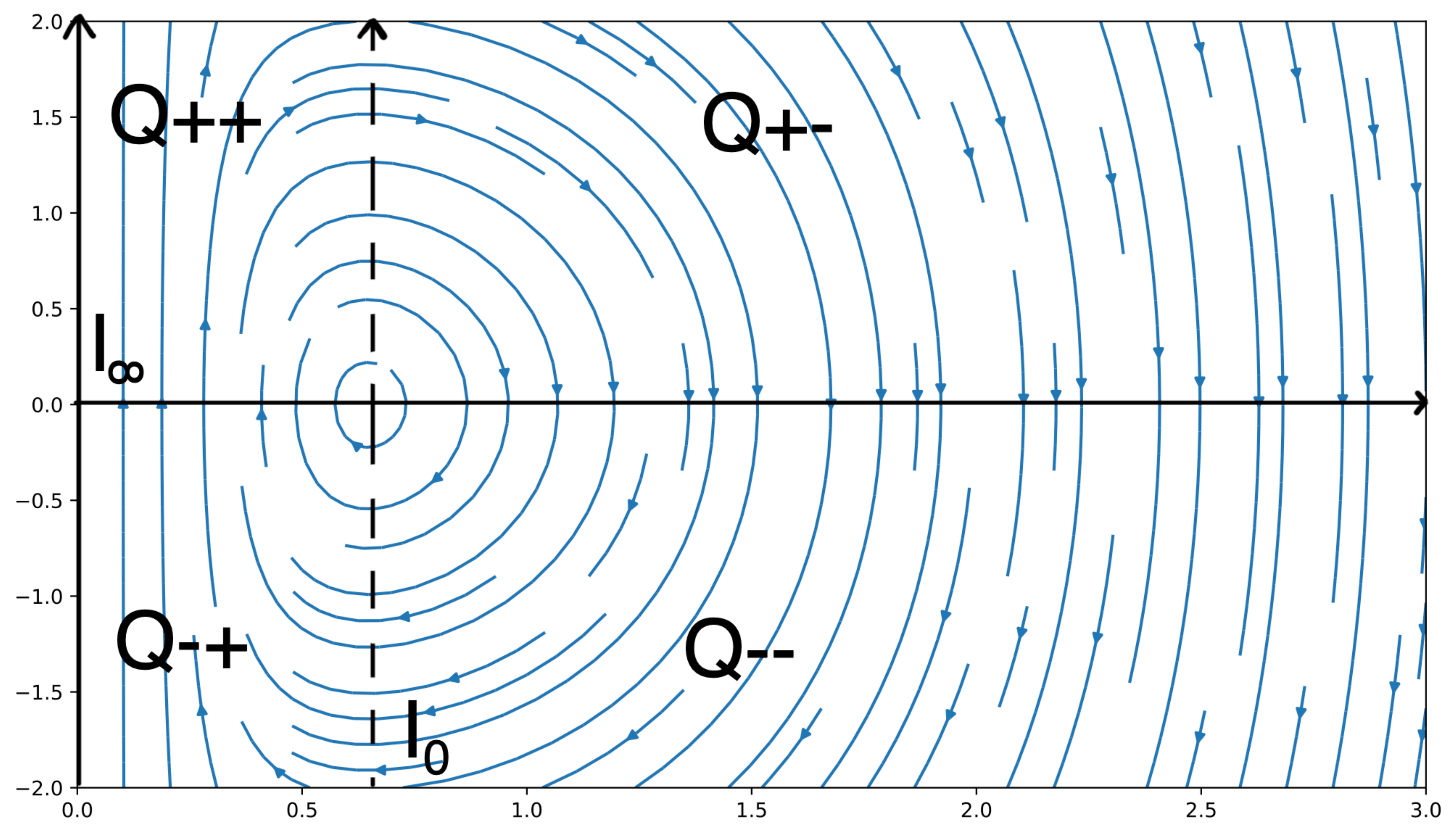}
  \end{tabular}
  \caption{Example of phase portraits of the solutions when $J \neq 0$.}
  \label{fig:phase-portrait-J-diff-0}
\end{figure}

\FloatBarrier

\section{The minimization problems}
\label{sec:minimization}

In this section, we study the variational properties of periodic
states. We start by establishing a Fourier rearrangement inequality
that will be useful later on. We then study the minimization of the
energy at fixed mass (and momentum) in various settings (periodic,
anti-periodic, focusing/defocusing nonlinearity). Finally, we
consider the minimization of the action over the Nehari manifold.

In addition to~\ref{item:h1}-\ref{item:h3}, we will sometime make use of some of the following additional assumptions on the nonlinearity.
\begin{enumerate}[start=4,label={(H\arabic*)}]
\item \label{item:h4} The function 
  \begin{equation} \label{eq:assumption-on-the-nonlinearity-h}
    h(s):=(sf(s)-2F(s))s^{-2},
  \end{equation}
  is strictly increasing on $(0, +\infty)$. 
  Moreover, $\lim \limits_{s \to 0} h(s)=0$ and $\lim \limits_{s \to 0} \frac{f(s)}{s}=0$.
\item \label{item:h5} There exist $M>0$ , $1<p<5$ and $s_0$ such that for all $s\geq s_0$ we have $|f(s)|\leq M s^p$.
\item \label{item:h6} For any $s>0$, the following inequality is satisfied:
  \begin{equation} \label{eq:assumption}
    s^2f''(s)>sf'(s)-f(s).
  \end{equation}
\item \label{item:h7} \label{} At infinity, we have
  \begin{equation} \label{eq:assumption-on-A(s)}
    \lim \limits_{s \to \infty}\left( \frac{f(s)}{s}-f'(s)\right)= - \infty.
  \end{equation}
\end{enumerate}

Most of these assumptions are related to the growth of the nonlinearity and are satisfied by sums of generic power nonlinearities. The main restriction may comes from~\ref{item:h5}, which imposes a mass-subcritical growth on the nonlinearity and is used for minimization of the energy on the mass constraint in the focusing case.


We will denote norms on $L^q(0,T)$ spaces by
\[
  \|u\|_{L^q}= \| u \| _{L^q(0,T)}= \left( \int_0^T |u|^q \right) ^{\frac{1}{q}},
\]
and the complex $L^2$ inner product by
\[
  (f,g)= \int_0^T f \bar{g} dx.
\]

We will be interested in spatially periodic solutions $\psi \in H^1_{loc}(\R) \cap P_T$, and anti-periodic solutions $\psi \in H^1_{loc}(\R) \cap A_T$, where
\[
  P_T=\{ f \in L^2_{loc} (\R): f(x+T) =f(x) \},
  \quad
  A_T= \{ f \in L^2 _{loc} (\R): f(x+T)=-f(x) \}.
\]


\subsection{A Fourier rearrangement inequality}

We start by presenting a lemma based on a Fourier rearrangement process that will be useful later on. This is a generalization of a result used in~\cite{GuLeTs17} in the cubic case.

\begin{lemma} \label{lem:new}
  Let $v \in H^1_{loc}(\R) \cap A_{\frac{T}{2}}$ and $p>1$ an odd integer. Then there exists $\tilde v \in H^1_{loc}(\R) \cap A_{\frac{T}{2}} $ such that:
  \[
    \tilde v (x) \in \R, \quad \|\tilde v\|_{L^2}=\|v\|_{L^2}, \quad \|\partial_x \tilde v \|_{L^2}=\|\partial_x v\|_{L^2}, \quad \|\tilde v\|_{L^{p+1}} \geq \|v\|_{L^{p+1}}.
  \]
\end{lemma}

\begin{proof}
  Since $v \in H^1_{loc}(\R) \cap A_{\frac{T}{2}}$, its Fourier series expansion contains only terms indexed by odd integers:
  \[
    v(x)= \sum \limits_{\underset{j odd}{j \in \Z }} v_j e^{i j \frac{2 \pi}{T}x}.
  \]
  We define $\tilde v$ by its Fourier series expansion
  \[
    \tilde v(x)= \sum \limits_{\underset{j odd}{j \in \Z }} \tilde v_j e^{i j \frac{2 \pi}{T}x}, \quad \tilde v_j := \sqrt{\frac{|v_j|^2+|v_{-j}|^2}{2}}.
  \]
  It is clear that $\tilde v (x) \in \R$, and by Plancherel formula, we have
  \[
    \|\tilde v\|_{L^2}=\|v\|_{L^2}, \quad \|\partial_x \tilde v \|_{L^2}=\|\partial_x v\|_{L^2},
  \]
  so all we have to prove is that $\|\tilde v\|_{L^{p+1}} \geq \|v\|_{L^{p+1}}$. We have 
  \[
    |v(x)|^2= \sum \limits_{\underset{j odd}{j \in \Z }} |v_j|^2 + \sum \limits_{\underset{n \geq 2 }{n \in 2 \N }} w_n e^{in\frac{2 \pi}{T}x}+\bar w_n e^{-in\frac{2 \pi}{T}x},
  \]
  where we have defined \[
    w_n= \sum \limits_{\underset{j,k odd}{j>k, j+k=n}} v_j \bar v_{-k}+ v_k \bar v_{-j}.
  \]
  Let $ N=\frac{p+1}{2}$. We start with
  \begin{align*}
    |v|^{p+1}&= \left( |v|^2 \right) ^{\frac{p+1}{2}}\\
             &= \left( \sum \limits_{\underset{j odd}{j \in \Z }} |v_j|^2 + \sum \limits_{\underset{n \geq 2 }{n \in 2 \N }} w_n e^{in\frac{2 \pi}{T}x}+\bar w_n e^{-in\frac{2 \pi}{T}x} \right)^ N,\\
             &= \sum \limits_{k=0}^N \binom{N}{k} \left(\sum \limits_{\underset{j odd}{j \in \Z }} |v_j|^2 \right)^{N-k} \left(\sum \limits_{\underset{n \geq 2 }{n \in 2 \N }} w_n e^{in\frac{2 \pi}{T}x}+\bar w_n e^{-in\frac{2 \pi}{T}x} \right)^k.
  \end{align*}
  We have 
  \begin{multline*}
    \left(\sum \limits_{\underset{n \geq 2 }{n \in 2 \N }} w_n e^{in\frac{2 \pi}{T}x}+\bar w_n e^{-in\frac{2 \pi}{T}x} \right)^k \\
    = \sum \limits_{s=0}^k \binom{k}{s} \sum \limits_{p_1} \cdots \sum \limits_{p_k} \bar w_{p_1}\cdots \bar w_{p_s} \cdot w_{p_s}w_{p_{s+1}} \cdots w_{p_k} e^{i(-p_1- \cdots -p_s+p_{s+1}+ \cdots +p_k)\frac{2 \pi}{T}x},\\
    = \sum \limits_{s=0}^k \binom{k}{s} \sum \limits_{p_1} \cdots \sum \limits_{p_k} \left( \prod_{l=1}^s \bar w_{p_l} e^{-i p_l \frac{2 \pi }{T}x} \right) \left( \prod_{l=s+1}^k w_{p_l} e^{i p_l \frac{2 \pi }{T}x} \right),
  \end{multline*}
  where we use the convention
  \[
    \prod_{l=1}^0 \bar w_{p_l} e^{-i p_l \frac{2 \pi }{T}x}=1, \quad \prod_{l=k+1}^k w_{p_l} e^{i p_l \frac{2 \pi }{T}x}=1.
  \]
  Then we have
  \begin{multline*}
    \frac{1}{T} \int_0^T |v|^{p+1} dx 
    \\= \frac{1}{T} \int_0^T \sum \limits_{k=0}^N \binom{N}{k} \left(\sum \limits_{\underset{j odd}{j \in \Z }} |v_j|^2 \right)^{N-k} \left(\sum \limits_{\underset{n \geq 2 }{n \in 2 \N }} w_n e^{in\frac{2 \pi}{T}x}+\bar w_n e^{-in\frac{2 \pi}{T}x} \right)^k dx ,\\
    = \sum \limits_{k=0}^N \binom{N}{k} \left(\sum \limits_{\underset{j odd}{j \in \Z }} |v_j|^2 \right)^{N-k} \cdot
    \\ \frac{1}{T} \int_0^T \sum \limits_{s=0}^k \binom{k}{s} \sum \limits_{p_1}...\sum \limits_{p_k} \left( \prod_{l=1}^s \bar w_{p_l} e^{-i p_l \frac{2 \pi }{T}x} \right) \left( \prod_{l=s+1}^k w_{p_l} e^{i p_l \frac{2 \pi }{T}x} \right)dx, \\
    = \sum \limits_{k=0}^N \binom{N}{k} \left(\sum \limits_{\underset{j odd}{j \in \Z }} |v_j|^2 \right)^{N-k} \sum \limits_{s=0}^k \binom{k}{s} \sum \limits_{p_1, ... ,p_n \in \sigma} \left( \prod_{l=1}^s \bar w_{p_l} \right) \left( \prod_{l=s+1}^k w_{p_l} \right),
  \end{multline*}
  where $ \sigma = \{ (p_1,...,p_n): \exists \alpha \in \{0,1\}^n
  : \sum \limits_{j} (-1)^{\alpha_j} p_j=0\} $, and where we have used the fact that for $n \in \N, n \neq 0$, we have 
  \[
    \int_0^T e^{in\frac{2 \pi }{T}x} dx=0.
  \]
  On the other hand, we observe that
  \begin{equation} \label{eq:w_n}
    w_n= \sum \limits_{\underset{j,k odd}{j>k, j+k=n}} \begin{pmatrix}
      v_j\\
      \tilde v_{-j}
    \end{pmatrix} . \begin{pmatrix}
      v_k\\
      \tilde v_k
    \end{pmatrix},
  \end{equation}
  where the . denotes the complex vector scalar product. Therefore,
  \begin{align*}
    |w_n| & \leq \sum \limits_{\underset{j,k odd}{j>k, j+k=n}} \left| \begin{pmatrix}
        v_j\\
        \tilde v_{-j}
      \end{pmatrix} \right| \left| \begin{pmatrix}
        v_k\\
        \tilde v_k
      \end{pmatrix} \right|= \sum \limits_{\underset{j,k odd}{j>k, j+k=n}} \sqrt{2 \tilde v _j^2} \sqrt{2 \tilde v_k^2}\\
          &= 2 \sum \limits_{\underset{j,k odd}{j>k, j+k=n}} \tilde v_j \tilde v_k = \tilde w_n,
  \end{align*}
  where by $\tilde w_n$, we denote the quantity defined similarly as in~\eqref{eq:w_n} for $ (\tilde v_j)$.
  Therefore,
  \begin{multline*}
    \frac{1}{T} \int_0^T |v|^{p+1} dx
    \\ \leq \sum \limits_{k=0}^N \binom{N}{k} \left(\sum \limits_{\underset{j odd}{j \in \Z }} |v_j|^2 \right)^{N-k} \sum \limits_{s=0}^k \binom{k}{s} \sum \limits_{p_1, ... ,p_n \in \sigma} \left( \prod_{l=1}^s |\bar w_{p_l}| \right) \left( \prod_{l=s+1}^k |w_{p_l}| \right),\\
    \leq \sum \limits_{k=0}^N \binom{N}{k} \left(\sum \limits_{\underset{j odd}{j \in \Z }} |v_j|^2 \right)^{N-k} \sum \limits_{s=0}^k \binom{k}{s} \sum \limits_{p_1, ... ,p_n \in \sigma} \left( \prod_{l=1}^s \tilde w_{p_l} \right) \left( \prod_{l=s+1}^k \tilde w_{p_l} \right),\\ 
    = \frac{1}{T} \int_0^T |\tilde v|^{p+1} dx,
  \end{multline*}
  which concludes the proof.
\end{proof}

\subsection{Minimization on the mass constraint}

We now consider our first set of variational problems. 
Let $m>0$. A common variational problem is to minimize the energy at fixed mass:
\begin{equation} \label{min1.1}
  \min \{ \mathcal{E}(u): u \in H^1_{loc}(\R)\cap P_T,\; M(u)=m \}. 
\end{equation}
Since the momentum is also conserved for~\eqref{eq:nls4}, it is natural to consider the problem with a further momentum constraint:
\begin{equation} \label{min1.2}
  \min \{ \mathcal{E}(u): u \in H^1_{loc}(\R) \cap P_T, \; M(u)=m, \; P(u)=0 \}.
\end{equation}
The minimization problems~\eqref{min1.1} and~\eqref{min1.2} seek to find functions $u$ which minimize the energy subject to the constraint that the mass is fixed and, in the case of~\eqref{min1.2}, the momentum is also zero.
Note that when we minimize the energy with fixed mass and fixed momentum $p \neq 0$ the problem is more complicated.
In our work we will only focus on the case $p=0$.
\subsubsection{The focusing case in $P_T$ }
Assume that $b>0$. 
\begin{proposition} \label{prop:focusing-case-existence-min}
  Assume that $f$ verifies~\ref{item:h1}-\ref{item:h3} and~\ref{item:h5}.
  For all $m>0$, the minimization problem~\eqref{min1.1} admits
  a real minimizer which is also a minimizer of the minimization problem~\eqref{min1.2}. The minimal energy is finite and negative.
\end{proposition}

\begin{proof}
  Without loss of generality, we can restrict the minimization to real valued non-negative functions. Indeed, if $u \in H^1_{loc}(\R) \cap P_T $, then $|u| \in H^1_{loc}(\R) \cap P_T $ and we have $ \| \partial _x |u| \| _{L^2} \leq \| \partial _x u \| _{L^2}.$
  This implies that~\eqref{min1.1} and~\eqref{min1.2} share the same minimizers. 
  
  Let us prove that the minimal energy is negative. To do so, let $\phi_{m,0} \equiv \sqrt{\frac{2m}{T}}$ be a test function. We have
  \[ 
    M(\phi_{m,0})=m, \quad \mathcal{E}(\phi_{m,0})= - \int_0^T F \left( \sqrt{\frac{2m}{T}} \right)dx= -TF \left( \sqrt{\frac{2m}{T}} \right)<0, 
  \] 
  where the last inequality holds because $F(z)>0$ for any $z\in\mathbb C$ by the assumptions on $f$. 
  
  Consider now a minimizing sequence $(u_n) \subset H^1_{loc}(\R) \cap P_T $ for~\eqref{min1.1}. We first prove that it is bounded in $H^1_{loc}(\R) \cap P_T.$ 
  To this aim, we rely on the Gagliardo-Nirenberg inequality: for any $u \in H^1_{loc}(\R) \cap P_T$, we have
  \[ 
    \|u\|^{p+1}_{L^{p+1}} \lesssim \|u_x\|^{\alpha (p+1)}_{L^2} \|u\|^{(1-\alpha)(p+1)}_{L^2} + \|u\|^{p+1}_{L^2},
  \]
  where $\alpha = \frac12 -\frac{1}{p+1}.$
  We also know that there exists $p>1$ such that 
  \[
    F(u)\leq |F(u)| \lesssim |u|^2+|u|^{p+1}.
  \] 
  Consequently, for any $u \in H^1_{loc}(\R) \cap P_T$, such that $M(u)=m$, we have
  \begin{align*}
    \mathcal{E}(u)&= \frac{1}{2} \|u_x\|^2_{L^2} - \int_0^T F(u) dx,\\
                  & \geq \frac{1}{2} \|u_x\|^2_{L^2}-C(\|u\|^2_{L^2} -\|u\|^{p+1}_{L^{p+1}}),\\
                  & \geq \frac{1}{2} \|u_x\|^2_{L^2}-Cm -C\|u_x\|^{\alpha(p+1)}_{L^2} m^{\frac{(1-\alpha)(p+1)}{2}}-Cm^{\frac{p+1}{2}},\\
                  &= \|u_x\|^2_{L^2} \left( \frac{1}{2} - C\|u_x\|^{\alpha(p+1)-2} m^{\frac{(1-\alpha)(p+1)}{2}} \right) - C(m^{\frac{p+1}{2}}-m).
  \end{align*}
  The previous inequality implies the boundedness of $\| \partial_x u_n \|_{L^2}$ when $1<p<5$. Indeed, by contradiction, we suppose that $\|\partial_x u_n\|_{L^2} \rightarrow \infty$. Since $1<p<5$, we have $\alpha(p+1)-2<0$, and this implies that $\|\partial_x u_n\|^{\alpha(p+1)-2}_{L^2} \rightarrow 0$, and therefore $\mathcal{E}(u_n) \rightarrow \infty$, which is a contradiction with the minimizing nature of $(u_n)$. Moreover, the same arguments show that if $1<p<5$, then the minimal energy is finite. Hence the sequence
  $(u_n)$ is bounded in $H^1_{loc}(\R) \cap P_T$.
  Therefore up to a subsequence, $(u_n)$ converges weakly in $H^1_{loc}(\R) \cap P_T$ and strongly in $L^2_{loc} \cap P_T$ and $L^{p+1}_{loc} \cap P_T$ towards $u_{\infty} \in H^1_{loc}(\R) \cap P_T.$ We now show that $(u_n)$ converges strongly towards $u_\infty$ in $H^1_{loc}(\R) \cap P_T$. 
  By weak convergence, we have 
  \[
    \|\partial_x u_{\infty}\|^2_{L^2} \leq \lim\limits_{n \rightarrow +\infty} \|\partial_x u_n\|^2_{L^2}.
  \]
  Up to a subsequence, we also have $F(u_n) \rightarrow F(u_\infty)$ almost everywhere. Moreover, we have
  \begin{align*}
    |F(u_n)| &\lesssim |u_n|^2+|u_n|^{p+1}\\
             &\lesssim \|u_n\|^2_{L^{\infty}} +\|u_n\|^{p+1}_{L^{\infty}}\\
             &\lesssim \|u_n\|^2_{H^1} +\|u_n\|^{p+1}_{H^1} \leq \max_{n\in\mathbb N}\{\|u_n\|^2_{H^1} +\|u_n\|^{p+1}_{H^1}\}<\infty.
  \end{align*}
  Then by the dominated convergence theorem we have
  \[
    \lim\limits_{n \rightarrow +\infty} \int_0^T F(u_n) dx = \int_0^T F(u)dx. 
  \]
  Combining the previous arguments, we obtain
  \[
    \mathcal{E}(u_{\infty}) \leq \lim\limits_{n \rightarrow +\infty}\mathcal{E}(u_n),\quad M(u_n)=m,
  \]
  which in turn implies 
  \[
    \|\partial_x u_{\infty}\|^2_{L^2} = \lim\limits_{n \rightarrow +\infty} \|\partial_x u_n\|^2_{L^2}.
  \]
  Therefore the convergence from $(u_n)$ to $u_{\infty}$ is also strong in $ H^1_{loc}(\R) \cap P_T.$
\end{proof}

\begin{proposition}
  \label{prop:focusing-case-P_T}
  Assume that $f$ verifies~\ref{item:h1}-\ref{item:h3} and~\ref{item:h4}-\ref{item:h7}.
  There exists $\tilde m>0$ such that if $m>\tilde{m}$, then the minimizer of \eqref{min1.1} is not a constant, the associated Lagrange multiplier verifies $a<0$, the minimizer is positive.
\end{proposition}

\begin{remark}
  In the cubic case, it is known that for small enough values of $m$, the minimizer of the energy functional in this case is the constant function. 
\end{remark}

\begin{proof} 
  Since $u_{\infty}$ is a minimizer of~\eqref{min1.1}, there exists a Lagrange multiplier $a \in \R $ such that
  \[
    -\mathcal{E}'(u_{\infty})+a M'(u_{\infty})=0\]
  that is \[ \partial_{xx} u_{\infty} +a u_{\infty} + b f(u_{\infty})=0.
  \]
  Multiplying by $u_{\infty} $ and integrating (recall that the functions considered are assumed to be real), we find that
  \[ 
    a= \frac{\|\partial_x u_{\infty}\|^2_{L^2}- b\int_0^T f(u_{\infty})u_{\infty} dx}{\|u_{\infty}\|^2_{L^2}}. 
  \]
  Note that 
  \begin{align*}
    \|\partial_x u_{\infty}\|^2_{L^2}- b\int_0^T f(u_{\infty})u_{\infty} dx &= 2 \mathcal{E}(u_{\infty}) + 2b \int_0^T F(u_{\infty}) dx - \int_0^T b f(u_{\infty}) u_{\infty} dx \\
                                                                            &= 2\mathcal{E}(u_{\infty}) + b\int_0^T ( 2F(u_{\infty}) -f(u_{\infty}) u_{\infty}) dx,
  \end{align*}
  where $\mathcal{E}(u_{\infty}) <0 $ and $2F(u_{\infty}) -f(u_{\infty}) u_{\infty}<0$, by the assumption on~\eqref{eq:assumption-on-the-nonlinearity-h}. Therefore, we have
  \[ 
    a<0.
  \]
  We introduce an auxiliary function 
  \[
    A(s)= \frac{4 \pi^2}{T^2}+b \left( \frac{f(s)}{s}-f'(s) \right).
  \]
  By assumption~\eqref{eq:assumption}, we have
  \[
    A'(s)=b \left( \frac{f'(s)s-f(s)}{s^2}-f''(s) \right) <0.
  \]
  Therefore $A$ is a decreasing function, from $\frac{4 \pi^2}{T^2}$ to $-\infty$ from assumption~\eqref{eq:assumption-on-A(s)}. Let $m^{*}$ be such that $A(m^{*})=0$ and define 
  \[
    \tilde m=\frac{T m^{*2}}{2}.
  \]
  We want to prove that if $m>\tilde{m}$, then $u_{\infty}$ is not constant.
  
  By contradiction, we assume that $u_{\infty}$ is constant for $m>\tilde m$. Then we necessarily have $u_{\infty} \equiv \sqrt{\frac{2m}{T}}$. The Lagrange multiplier can also be computed and we find 
  \[
    a= \frac{- b\int_0^T f\left(\sqrt{\frac{2m}{T}} \right)\sqrt{\frac{2m}{T}} dx}{2m}= -b f\left(\sqrt{\frac{2m}{T}} \right)\sqrt{\frac{T}{2m}}.
  \]
  Since $u_{\infty}$ is supposed to be a constrained minimizer for~\eqref{min1.1}, the operator 
  \[ 
    -\partial_{xx} -a - b f'(u_{\infty})= - \partial_{xx} + b \sqrt{\frac{T}{2m}} f\left(\sqrt{\frac{2m}{T}} \right) - b f'\left(\sqrt{\frac{2m}{T}} \right),
  \]
  must have Morse Index at most 1, i.e, at most 1 negative eigenvalue. The eigenvalues are given for $n \in \Z$ by the following formula:
  \[ 
    \left(\frac{2 \pi n}{T} \right)^2+ b \sqrt{\frac{T}{2m}} f\left(\sqrt{\frac{2m}{T}} \right) - b f'\left(\sqrt{\frac{2m}{T}} \right), \quad \quad n \in \Z. 
  \]
  If $n=0$, the eigenvalue is negative:
  \[
    b \sqrt{\frac{T}{2m}} f\left(\sqrt{\frac{2m}{T}} \right) - b f'\left(\sqrt{\frac{2m}{T}} \right)<0.
  \]
  Indeed as $\frac{f(s)}{s} $ is an increasing function we have that for all $s>0$, $ \left(\frac{f(s)}{s}\right)'= \frac{f'(s)s-f(s)}{s^2}>0$.
  If $n=1$ the eigenvalue is of the form:
  \[ 
    \frac{4 \pi^2}{T^2} + b \sqrt{\frac{T}{2m}} f\left(\sqrt{\frac{2m}{T}} \right) -b f'\left(\sqrt{\frac{2m}{T}} \right)=A\left(\sqrt{\frac{2m}{T}} \right). 
  \]
  Recall that $A\left(\sqrt{\frac{2m}{T}}\right)$ is non-negative if and only if $\sqrt{\frac{2m}{T}} \leq m^{*}$ which is equivalent to $m \leq \tilde m$ which gives the contradiction. Therefore when $m> m^*$ the minimizer $u_{\infty}$ is not constant, which concludes the proof.
\end{proof}

\subsubsection{The defocusing case in $P_T$}
Assume that $b<0$. 
\begin{proposition} \label{prop:defocusing-case-P_T}
  Assume that $f$ verifies~\ref{item:h1}-\ref{item:h3}.
  For all $m \in (0,\infty)$ the constrained minimization problems~\eqref{min1.1} and~\eqref{min1.2} have the same unique (up to phase shift) minimizer, which is the constant function $u_{\infty} \equiv \sqrt{\frac{2m}{T}}.$
\end{proposition}

\begin{proof}
  Consider a minimizing sequence $(u_n) \subset H^1_{loc}(\R) \cap P_T $ for~\eqref{min1.1}. We first prove that it is bounded in $H^1_{loc}(\R) \cap P_T.$ 
  We have 
  \[
    \mathcal{E}(u_n) = \frac{1}{2} \|\partial_x u_n\|^2_{L^2} -b \int_0^T F(u_n) dx\geq 0.
  \]
  By contradiction, we suppose that $\|\partial_x u_n\|_{L^2} \rightarrow \infty$. Therefore $\mathcal{E}(u_n) \rightarrow \infty$, which is a contradiction with the minimizing nature of $(u_n)$. Moreover, the same argument show that the minimal energy is finite. Hence the sequence
  $(u_n)$ is bounded in $H^1_{loc}(\R) \cap P_T$.
  Therefore up to a subsequence, $(u_n)$ converges weakly in $H^1_{loc}(\R) \cap P_T$ and strongly in $L^2_{loc} \cap P_T$ and $L^{p+1}_{loc} \cap P_T$ towards $u_{\infty} \in H^1_{loc}(\R) \cap P_T.$ As in the proof of Proposition~\ref{prop:focusing-case-existence-min} we have that $(u_n)$ converges strongly towards $u_{\infty} $ in $H^1_{loc}(\R) \cap P_T$. As for the focusing case, for any $v\in H_{loc}^1\cap P_T$, we have
  \[ 
    M(|v|)=M(v), \quad \mathcal{E}(|v|) \leq \mathcal{E}(v),
  \]
  therefore we may assume that $u_{\infty} \geq 0$. 
  
  As in the proof of Proposition~\ref{prop:focusing-case-P_T}, we know that there exists a Lagrange multiplier $a$ such that
  \begin{equation*}
    -\mathcal{E}'(u_{\infty})+a M'(u_{\infty})=0,
  \end{equation*}
  i.e. $u_{\infty}$ satisfies the ordinary differential equation~\eqref{eq:ode4}.
  Hence $a$ might be explicitly expressed in the following way: 
  \[ 
    a= \frac{\|\partial_x u_{\infty}\|^2_{L^2}- b\int_0^T f(u_{\infty})u_{\infty} dx}{\|u_{\infty}\|^2_{L^2}}. 
  \] 
  Since $b<0$, we have $a>0$.
  In this case we know that the phase portrait for real valued solutions of~\eqref{eq:ode4} is given in Figure~\ref{fig:phase-portrait-def-J=0}.
  
  The only solutions of~\eqref{eq:ode4} that do not change sign are the constant functions $\pm\sqrt{\frac{2m}{T}}$. As a consequence, there exists $\theta \in \R$ such that
  \[
    u_{\infty} =e^{i\theta}\sqrt{\frac{2m}{T}},
  \]
  which concludes the proof. 
\end{proof}

\begin{remark}
  Under the assumptions of Proposition~\ref{prop:defocusing-case-P_T}, the minimizer is $u_{\infty}\equiv \sqrt{\frac{2m}{T}}$ (up to phase shift), and therefore the associated Lagrange multiplier is given by 
  \[
    a= -bf\left(\sqrt{\frac{2m}{T}} \right)\sqrt{\frac{T}{2m}}.
  \]
  The eigenvalues of the associated linearized operator
  \[
    -\partial_{xx} -a - b f'(u_{\infty})= - \partial_{xx} + b\sqrt{\frac{T}{2m}} f\left(\sqrt{\frac{2m}{T}} \right) -bf'\left(\sqrt{\frac{2m}{T}} \right)
  \]
  are given for $n \in \Z$ by the following formula:
  \[
    \left(\frac{2 \pi n}{T} \right)^2+b\left(\sqrt{\frac{T}{2m}} f\left(\sqrt{\frac{2m}{T}} \right) -f'\left(\sqrt{\frac{2m}{T}} \right)\right). 
  \]
  Since $b<0$ and $f(s)s^{-1}-f'(s)< 0$, if we assume in addition~\ref{item:h4}, we remark that the eigenvalues are all positive.
\end{remark}


We will also consider the variational problem restricted to anti-symmetric functions: 
\begin{equation} \label{min2.1}
  \min \{ \mathcal{E}(u):u \in H^1_{loc}(\R) \cap A_{\frac{T}{2}},\; M(u)=m\}.
\end{equation}

\subsubsection{The defocusing case in $A_T$} 

Assume $b<0$. 

In this section, we restrict ourselves to the sum of several powers.
\begin{proposition} 
  Let $f(u)=\sum_{j=1}^{N}|u|^{p_j-1}u$, where $N>1$ and $1<p_j<\infty$, for $j=1,\dots,N$. There exists a unique (up to phase shift and complex conjugate) minimizer of~\eqref{min2.1}. It is the plane wave $u_{\infty} \equiv \sqrt{\frac{2m}{T}} e^{\frac{i \pi x}{T}}$.
\end{proposition}

\begin{proof}
  Denote the supposed minimizer by $w(x)= \sqrt{\frac{2m}{T}} e^{\pm \frac{i \pi x}{T}}$. Let $v \in H^1_{loc}(\R) \cap A_{\frac{T}{2}}$ such that: $M(v)=m$ and $v \not\equiv e^{i \theta} w $ ($\theta \in \R$). Since $v \in H^1_{loc}(\R) \cap A_{\frac{T}{2}} $, $v$ must have 0 mean value. Recall that in this case $v$ verifies the Poincar\'e-Wirtinger inequality 
  \[
    \|v\|_{L^2} \leq \frac{T}{2 \pi} \|v'\|_{L^2},
  \]
  and that the optimizers of the Poincar\'e-Wirtinger inequality are of the form $c e^{\pm \frac{i \pi x}{T}}, c \in \C$. This implies that 
  \[
    \| \partial_x w \|^2_{L^2}= \frac{8 \pi ^2}{T^2} M(w)=\frac{8 \pi ^2}{T^2} M(v)<\|\partial_x v\|^2_{L^2}.
  \]
  We will prove now that $\int_0^T F(w) dx \leq \int_0^T F(v) dx$.
  We have 
  \begin{align*}	
    &\int_0^T \sum_{j=1}^{N} \left( \frac{1}{p_j+1} |w|^{p_j+1} \right) dx, \\
    &=\int_0^T \sum_{j=1}^{N}\left( \left(\frac{1}{p_j+1}\right) \left|\sqrt{\frac{2m}{T}}\right|^{p_j+1} \right) dx, \\
    &= T \sum_{j=1}^{N} \left(\frac{1}{p_j+1}\right) \left( T^{-\frac{p_j+1}{2}} 2^{\frac{p_j+1}{2}} m^{\frac{p_j+1}{2}} \right),\\
    &= T \sum_{j=1}^{N}\left(\frac{1}{p_j+1}\right) \left( T^{-\frac{p_j+1}{2}} \|v\|^{p_j+1}_{L^2} \right),\\
    & \leq T \sum_{j=1}^{N} \left(\frac{1}{p_j+1}\right) \left( T^{-\frac{p_j+1}{2}} T^{\frac{p_j-1}{2}} \|v\|^{p_j+1}_{L^{p_j+1}}\right),
  \end{align*}
  where the last inequality came from H\"older inequality:
  \[
    \|v\|^{p +1} _{L^2} \leq T^{\frac{p-1}{2}} \|v\|^{p +1}_{L^ {p+1}}, \quad \frac{1}{\frac{p +1}{2}}+\frac{1}{\frac{p+1}{p-1}}=1.
  \]
  Therefore we have 
  \[
    \sum_{j=1}^{N} \frac{1}{p_j+1}\|w\|^{p_j+1}_{L^{p_j+1}} \leq \sum_{j=1}^{N} \frac{1}{p_j+1} \|v\|^{p_j+1}_{L^{p_j+1}} 
  \]
  which implies that 
  \[
    \mathcal{E}(w)< \mathcal{E}(v),
  \]
  which concludes the proof.
\end{proof}


\subsection{Minimization on the Nehari manifold}
In this section we restrict ourselves to the nonlinearity of the form $f(u)= |u|^{p-1}u$, with $p>1$. We define the functional $S: H^1_{loc}(\R) \cap P_T\to \R$ by setting for $u \in H^1_{loc}(\R)$
\[ 
  S(u):= \frac{1}{2} \| \partial_x u\|^2_{L^2} - \frac{a}{2} \|u\|^2_{L^2}-\frac{b}{p+1} \|u\|^{p+1}_{L^{p+1}}.
\]
It is standard that $S$ is of class $C^2$. The Fr\'echet derivative of $S$ at $u$ is given by
\[
  S'(u)= -u_{xx} -a u -b |u|^{p-1}u.
\]
Therefore, $u$ is a solution of the ordinary differential equation~\eqref{eq:ode4} if and only if $S'(u)=0$. Let $I(u)= \| \partial_x u\|^2_{L^2} - a \|u\|^2_{L^2}-b \|u\|^{p+1}_{L^{p+1}}$. The set 
\[
  \{ u \in H^1_{loc}(\R): u \neq 0, I(u)=0 \}
\]
is called Nehari manifold.
We are interested in the minimization problems on the Nehari manifold:
\begin{equation} \label{min:Nehari-Manifold-periodic}
  \min \{ S(u): u \in H^1_{loc}(\R) \cap P_T,u \neq 0, I(u)=0\},
\end{equation}
and 
\begin{equation} \label{min:Nehari-Manifold-anti-periodic}
  \min \{ S(u): u \in H^1_{loc}(\R) \cap A_{\frac{T}{2}},u \neq 0, I(u)=0\}.
\end{equation}

The minimization problem on the Nehari manifold has been studied in numerous works. In this regard, we mention the work of Szulkin and Weth~\cite{SzWe10}, the work of Pankov~\cite{Pa07} and Pankov and Zhang~\cite{PaZh08} for the discrete nonlinear Schr\"odinger equation. 
We also mention the work of Hayashi~\cite{Ha21} on the nonlinear Schr\"odinger equation of derivative type and the work of Colin and Watanabe~\cite{CoWa18} on the nonlinear Klein-Gordon-Maxwell type system.

\begin{remark}
  The interest for minimization over the Nehari manifold is that it is a natural constraint. Indeed, assume that $u\neq 0$ is a minimizer.	We have 
  \begin{align*}
    I'(u)&= -2u_{xx} - 2a u - b(p+1) |u|^{p-1}u,\\
         &= -2u_{xx} - 2a u - 2b |u|^{p-1}u+ b(2-(p+1))|u|^{p-1}u,\\
         &= 2 S'(u) -b(p-1)|u|^{p-1}u.
  \end{align*}
  Moreover $I(u)=<S'(u),u>$, therefore 
  \begin{align*}
    <I'(u),u> &= 2<S'(u),u>- b(p-1) \|u\|^{p+1}_{L^{p+1}},\\
              &= -b(p-1) \|u\|^{p+1}_{L^{p+1}} \neq 0.
  \end{align*}
  On the other hand, if $S'(u)= \lambda I'(u)$, this implies that
  \[
    0=<S'(u),u>=\lambda <I'(u),u>.
  \] 
  Since $<I'(u),u> \neq 0 $, this implies
  \[
    \lambda =0.
  \]
  Therefore the minimizer $u$ verifies $S'(u)=0$, so it is a solution of the ordinary differential equation~\eqref{eq:ode4}.
\end{remark}

\subsubsection{The focusing case in $P_T$} 

Let $b>0$ and $a<0$. We have the following lemma.
\begin{lemma}
  The minimum of~\eqref{min:Nehari-Manifold-periodic} is finite and there exists a real minimizer solution of~\eqref{eq:ode4}.
\end{lemma}

\begin{proof}
  Consider a minimizing sequence $(u_n) \subset H^1_{loc}(\R) \cap P_T$ for~\eqref{min:Nehari-Manifold-periodic}. We have $I(u_n)=0$, therefore 
  \begin{equation*} 
    S(u_n)=S(u_n)-\frac{1}{p+1} I(u_n)= \left( \frac{1}{2}- \frac{1}{p+1} \right) \left( \|\partial _{x} u_n\|^2_{L^2}-a \|u_n\|^2_{L^2}\right).
  \end{equation*}
  We have the boundedness of the sequence $(u_n)$ in $H^1_{loc}(\R) \cap P_T$. Indeed, by contradiction we suppose that $\|u_n\|^2_{L^{2}} \to \infty$, or $\|\partial _{x} u_n\|^2_{L^2} \to \infty $, therefore $S(u_n) \rightarrow \infty$, which is a contradiction with the minimizing nature of $(u_n)$.
  Therefore up to a subsequence, $(u_n)$ converges weakly in $ H^1_{loc}(\R) \cap P_T$ and strongly in $ L^2_{loc} \cap P_T$ and $ L^{p+1}_{loc} \cap P_T$ towards $u_{\infty} \in H^1_{loc}(\R) \cap P_T$.
  By the weak convergence we have
  \[
    \| \partial_x u_{\infty} \|_{L^2} \leq \lim \limits_{n \to \infty} \inf \| \partial_x u_n\|_{L^2},
  \]
  then 
  \[
    \| \partial_x u_{\infty} \|^2_{L^2} -a \|u_{\infty}\| ^2_{L^2} \leq \lim \limits_{n \to \infty} \inf \left( \| \partial_x u_{n} \|^2_{L^2} -a \|u_n\| ^2_{L^2} \right).
  \]
  Therefore
  \[
    S(u_{\infty})- \frac{1}{p+1} I(u_{\infty}) \leq \lim \limits_{n \to \infty} \inf S(u_n).
  \]
  On the other hand we have 
  \begin{align*}
    I(u_{\infty}) &= \| \partial_x u_{\infty}\|^2_{L^2} -a \|u_{\infty}\|^2_{L^2}-b \|u_{\infty}\|^{p+1}_{L^{p+1}}\\
                  & \leq \lim \limits_{n \to \infty} \inf \left( \| \partial_x u_{n} \|^2_{L^2} -a \|u_n\| ^2_{L^2} \right) - b \lim \limits_{n \to \infty} \|u_n\|^{p+1}_{L^{p+1}}\\
                  & \leq \lim \limits_{n \to \infty} I(u_n)=0.
  \end{align*}
  Then 
  \[ I(u_{\infty}) \leq 0,
  \]
  and this implies that 
  \[
    S(u_{\infty}) \leq S(u_{\infty})- \frac{1}{p+1} I(u_{\infty}) \leq \lim \limits_{n \to \infty} \inf S(u_n).
  \]
  The graph of $I(tu)$ is given in the Figure~\ref{fig:I(ut)-foc}. We know that $I( u_{\infty}) \leq 0$, hence there exists $t_0 \leq 1$ such that $I(t_0u_{\infty})=0$.
  We have
  \begin{align*}
    S(t_0u_{\infty})&=S(t_0u_{\infty})-\frac{1}{p+1}I(t_0u_{\infty}),\\
                    &= \left( \frac{1}{2} -\frac{1}{p+1} \right) (\| \partial_x u_{\infty}\|^2_{L^2} -a \|u_{\infty}\|^2_{L^2}),\\
                    & \leq \left( \frac{1}{2} -\frac{1}{p+1} \right) (\| \partial_x u_{\infty}\|^2_{L^2} -a \|u_{\infty}\|^2_{L^2}),\\
                    &= S(u_{\infty})- \frac{1}{p+1} I(u_{\infty}) \leq \lim \limits_{n \to \infty} \inf S(u_n).
  \end{align*}
  Therefore 
  \[
    S(t_0 u_{\infty}) \leq \lim \limits_{n \to \infty} \inf S(u_n), \quad I(t_0 u_{\infty})=0,
  \]
  which implies the existence of the minimizer.
  
  Moreover, without loss of generality, we can restrict the minimization to real-valued non-negative functions. Indeed if $(t_0 u_{\infty}) \in H^1_{loc}(\R) \cap P_T$, then $|t_0 u_{\infty}| \in H^1_{loc}(\R) \cap P_T$ and we have $ \| \partial _x |t_0 u_{\infty}| \| _{L^2} \leq \| \partial _x (t_0 u_{\infty}) \| _{L^2}$. This implies that $S(|t_0 u_{\infty}|) \leq S(t_0 u_{\infty})$ and $I(|t_0 u_{\infty}|) \leq I(t_0 u_{\infty})=0$. As before, we know from the graph of $I(tu)$ given in the Figure~\ref{fig:I(ut)-foc} that there exists $t_2\leq 1$ such that $I(|t_2 u_{\infty}|)=0$ with $S(|t_2 u_{\infty}|) \leq \lim \limits_{n \to \infty} \inf S(u_n)$ which implies that the minimizer is real. 
  
  \begin{figure}[htbp!]
    \centering
    \begin{tabular}{cc}
      \includegraphics[width=.31\textwidth]{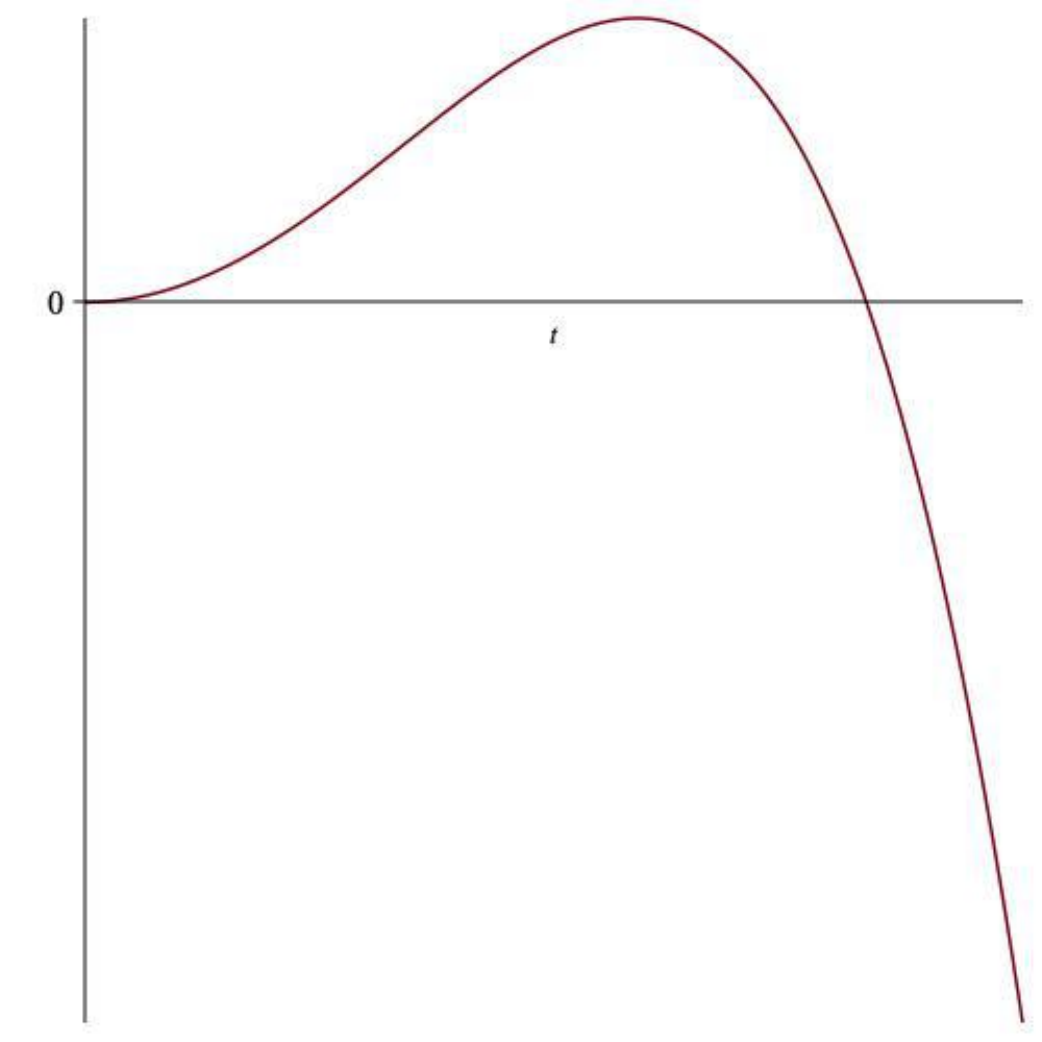}
    \end{tabular}
    \caption{$I(tu)$ as a function of $t$ in the focusing case.}
    \label{fig:I(ut)-foc}
  \end{figure}
  
\end{proof}

\subsubsection{The defocusing case in $P_T$}

Let $b<0$ and $a>0$. We have the following lemma.
\begin{lemma}
  The minimum of~\eqref{min:Nehari-Manifold-periodic} is finite and there exists a unique (up to phase shift) minimizer solution of~\eqref{eq:ode4} which is the constant function $u_{\min} \equiv \left({-\frac{a}{b}}\right)^{\frac{1}{p-1}}$.
\end{lemma}

\begin{proof}
  Consider a minimizing sequence $(u_n) \subset H^1_{loc}(\R) \cap P_T$ for~\eqref{min:Nehari-Manifold-periodic}.
  We know from the H\"older inequality that 
  \[
    \|v\|^{p +1} _{L^2} \leq T^{\frac{p-1}{2}} \|v\|^{p +1}_{L^ {p+1}}, \quad \frac{1}{\frac{p +1}{2}}+\frac{1}{\frac{p+1}{p-1}}=1.
  \]
  Thus, since $I(u_n)=0$, we have
  \[
    0 \leq \|\partial _{x} u_n\|^2_{L^2}=a \|u_n\|^2_{L^2}+b \|u_n\|^{p+1}_{L^{p+1}} \leq aT^{\frac{p-1}{p+1}}\|u_n\|^2_{L^{p+1}}+b \|u_n\|^{p+1}_{L^{p+1}} 
  \]
  As a result the last term in the inequality is positive with $b<0$ then $(u_n)$ is bounded in $L^{p+1}_{loc} \cap P_T$. Moreover, with the H\"older inequality we have then the boundedness of $(u_n)$ in $L^{2}_{loc} \cap P_T$. Finally as $I(u_n)=0$ we have the boundedness of the sequence $(u_n)$ in $H^1_{loc}(\R) \cap P_T$.
  Therefore up to a subsequence, $(u_n)$ converges weakly in $ H^1_{loc}(\R) \cap P_T$ and strongly in $ L^2_{loc} \cap P_T$ and $ L^{p+1}_{loc} \cap P_T$ towards $u_{\infty} \in H^1_{loc}(\R) \cap P_T$.
  By the weak convergence we have
  \[
    \|\partial _{x} u_{\infty} \|_{L^2} \leq \lim \limits_{n \to \infty} \inf \|\partial _{x} u_n\|_{L^2},
  \]
  then 
  \[
    \| \partial_x u_{\infty} \|^2_{L^2} -a \|u_{\infty}\| ^2_{L^2} \leq \lim \limits_{n \to \infty} \inf \left( \| \partial_x u_{n} \|^2_{L^2} -a \|u_n\| ^2_{L^2} \right).
  \]
  Therefore
  \[
    S(u_{\infty})- \frac{1}{p+1} I(u_{\infty}) \leq \lim \limits_{n \to \infty} \inf S(u_n).
  \]
  On the other hand we have 
  \begin{align*}
    I(u_{\infty}) &= \| \partial_x u_{\infty}\|^2_{L^2} -a \|u_{\infty}\|^2_{L^2}-b \|u_{\infty}\|^{p+1}_{L^{p+1}}\\
                  & \leq \lim \limits_{n \to \infty} \inf \left( \| \partial_x u_{n} \|^2_{L^2}\right) -a \lim \limits_{n \to \infty} \|u_n\| ^2_{L^2} - b \lim \limits_{n \to \infty} \|u_n\|^{p+1}_{L^{p+1}}\\
                  & \leq \lim \limits_{n \to \infty} I(u_n)=0.
  \end{align*}
  Then 
  \[ I(u_{\infty}) \leq 0,
  \]
  and this implies that 
  \[
    S(u_{\infty}) \leq S(u_{\infty})- \frac{1}{p+1} I(u_{\infty}) \leq \lim \limits_{n \to \infty} \inf S(u_n).
  \]
  The graph of $I(tu)$ is given in the Figure~\ref{fig:I(ut)-def}. Since $I( u_{\infty}) \leq 0$, there exists $t_0>1$ such that $I(t_0u_{\infty})=0$. Observe also that
  \[
    \| \partial_x u_{\infty}\|^2_{L^2} -a \|u_{\infty}\|^2_{L^2}\leq b \|u_{\infty}\|^{p+1}_{L^{p+1}}<0.
  \]
  We have
  \begin{align*}
    S(t_0u_{\infty})&=S(t_0u_{\infty})-\frac{1}{p+1}I(t_0u_{\infty}),\\
                    &= \left( \frac{1}{2} -\frac{1}{p+1} \right)t_0^2\left( \| \partial_x u_{\infty}\|^2_{L^2} -a \|u_{\infty}\|^2_{L^2}\right),\\
                    & \leq \left( \frac{1}{2} -\frac{1}{p+1} \right) \left( \| \partial_x u_{\infty}\|^2_{L^2} -a \|u_{\infty}\|^2_{L^2}\right),\\
                    &= S(u_{\infty})- \frac{1}{p+1} I(u_{\infty}) \leq \lim \limits_{n \to \infty} \inf S(u_n).
  \end{align*}
  Therefore 
  \[
    S(t_0 u_{\infty}) \leq \lim \limits_{n \to \infty} \inf S(u_n), \quad I(t_0 u_{\infty})=0,
  \]
  which implies the existence of the minimizer.
  As in the focusing case and without loss of generality, we can prove that the minimizer is real-valued, non negative and solution of the ordinary differential equation~\eqref{eq:ode4}. The only such solution of~\eqref{eq:ode4} is the constant functions $u_{\min} \equiv \left({-\frac{a}{b}}\right)^{\frac{1}{p-1}}$ with $I(u_{\min})=0$, which concludes the proof. 

  \begin{figure}[htbp!]
    \centering
    \begin{tabular}{cc}
      \includegraphics[width=.31\textwidth]{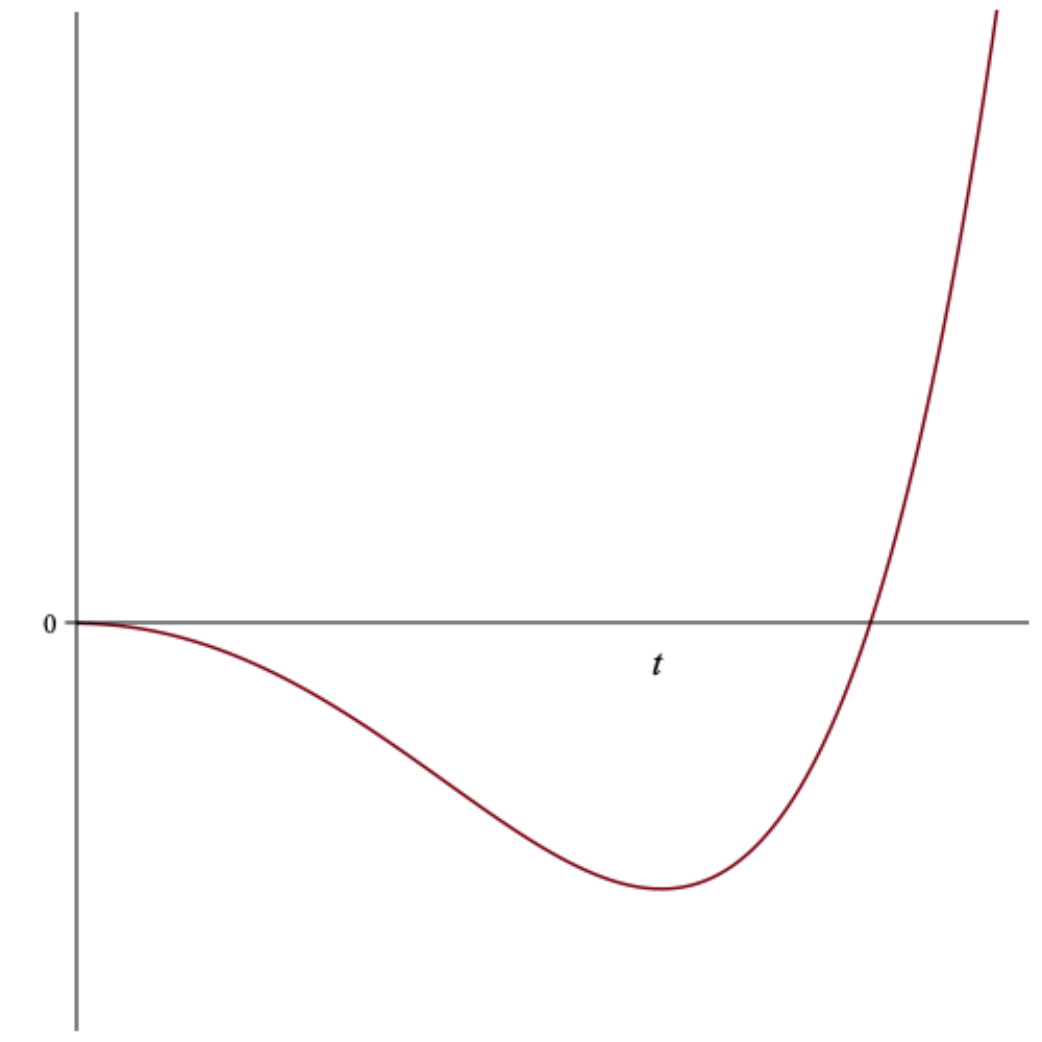}
    \end{tabular}
    \caption{$I(tu)$ as a function of $t$ in the defocusing case.}
    \label{fig:I(ut)-def}
  \end{figure}
\end{proof}


\subsubsection{The focusing case in $A_T$}
\label{sec:nehari-anti}

Assume $b>0$ and $a<\frac{4\pi^2}{T^2}$.
\begin{lemma}
  The minimum of~\eqref{min:Nehari-Manifold-anti-periodic} is finite.
\end{lemma}

\begin{proof}
  Consider a minimizing sequence $(u_n) \subset H^1_{loc}(\R) \cap A_{\frac{T}{2}}$ for~\eqref{min:Nehari-Manifold-anti-periodic}. We have $I(u_n)=0$, therefore 
  \begin{equation} \label{eq:S(u_n)}
    S(u_n)=S(u_n)-\frac{1}{p+1} I(u_n)= \left( \frac{1}{2} - \frac{1}{p+1} \right) \left( \|\partial _{x} u_n\|^2_{L^2}-a \|u_n\|^2_{L^2}\right).
  \end{equation}
  We will distinguish between two cases whether $a<0$ or $a>0$. 
  In the first case, as in the periodic case, we can directly conclude by contradiction with the minimizing nature of $(u_n)$ that it is bounded in $H^1_{loc}(\R) \cap A_{\frac{T}{2}}$.
  In the second case, we suppose that $a>0$. Since $u_n \in A_{\frac{T}{2}}$, $u_n$ must have 0 mean value. In that case $u_n$ verifies the Poincar\'e-Wirtinger inequality:
  \[
    \|u_n\|_{L^2} \leq \frac{T}{2 \pi } \| \partial _x u_n \|_{L^2}.
  \]
  Replacing in~\eqref{eq:S(u_n)}, we obtain that
  \[
    S(u_n) \geq \left( \frac{1}{2} - \frac{1}{p+1} \right) \left( \frac{4 \pi ^2}{T^2}-a \right) \|u_n\|^2_{L^2}.
  \]
  Then by the same arguments as in the first case we can prove that $(u_n)$ is bounded in $H^1_{loc}(\R) \cap A_{\frac{T}{2}}$ if 
  \[
    a< \frac{4 \pi^2}{T^2}.
  \]
  Therefore up to a subsequence, $(u_n)$ converges weakly in $ H^1_{loc}(\R) \cap A_{\frac{T}{2}}$ and strongly in $ L^2_{loc} \cap A_{\frac{T}{2}}$ and $ L^{p+1}_{loc} \cap A_{\frac{T}{2}}$ towards $u_{\infty} \in H^1_{loc}(\R) \cap A_{\frac{T}{2}}$.
  By the weak convergence we have
  \[
    \|u_{\infty} \|_{H^1} \leq \lim \limits_{n \to \infty} \inf \|u_n\|_{H^1}.
  \]
  If $a<0$, by the equivalence of the norms we have
  \[
    \| \partial_x u_{\infty} \|^2_{L^2} -a \|u_{\infty}\| ^2_{L^2} \leq \lim \limits_{n \to \infty} \inf \left( \| \partial_x u_{n} \|^2_{L^2} -a \|u_n\| ^2_{L^2} \right).
  \]
  And if $a>0$, by the strong convergence in $ L^2_{loc} \cap A_{\frac{T}{2}}$ we also have the above inequality.
  Therefore
  \[
    S(u_{\infty})- \frac{1}{p+1} I(u_{\infty}) \leq \lim \limits_{n \to \infty} \inf S(u_n).
  \]
  On the other hand we have 
  \begin{align*}
    I(u_{\infty}) &= \| \partial_x u_{\infty}\|^2_{L^2} -a \|u_{\infty}\|^2_{L^2}- \|u_{\infty}\|^{p+1}_{L^{p+1}}\\
                  & \leq \lim \limits_{n \to \infty} \inf \left( \| \partial_x u_{n} \|^2_{L^2} -a \|u_n\| ^2_{L^2} \right) - \lim \limits_{n \to \infty} \|u_n\|^{p+1}_{L^{p+1}}\\
                  & \leq \lim \limits_{n \to \infty} I(u_n)=0.
  \end{align*}
  Then 
  \[ I(u_{\infty}) \leq 0,
  \]
  and this implies that 
  \[
    S(u_{\infty}) \leq S(u_{\infty})- \frac{1}{p+1} I(u_{\infty}) \leq \lim \limits_{n \to \infty} \inf S(u_n).
  \]
  As in the periodic case with the Figure~\ref{fig:I(ut)-foc} we can prove that there exists $t_0<1$ such that $I(t_0u_{\infty})=0$ and $S(t_0 u_{\infty}) \leq \lim \limits_{n \to \infty} \inf S(u_n)$ which implies the existence of the minimizer.
\end{proof}

We now consider the following intermediate minimization problem:
\begin{equation} \label{min:prob-norm}
  \min \left\{ \left( \frac{1}{2}- \frac{1}{p+1} \right) \|v\|^2_{H^1} :v \neq 0, I(v)\leq 0, v \in H^1_{loc}(\R) \cap A_{\frac{T}{2}} \right\}.
\end{equation}
We have the following lemma.
\begin{lemma}
  The minimization problems~\eqref{min:Nehari-Manifold-anti-periodic} and~\eqref{min:prob-norm} share the same minimizers. Moreover, when $p$ is an odd integer, there exists a real minimizer.
\end{lemma}

\begin{proof}
  Let	
  \[
    m_1:=\min \{ S(u): u \in H^1_{loc}(\R) \cap A_{\frac{T}{2}},u \neq 0, I(u)=0\},
  \]
  and 
  \[
    m_2:= \min \{ \left( \frac{1}{2}- \frac{1}{p+1} \right) \|v\|^2_{H^1} : v \in H^1_{loc}(\R) \cap A_{\frac{T}{2}}, v\neq 0, I(v)\leq 0 \}.
  \]
  We will prove that $m_1=m_2$.
  Let $u$ be such that $m_1$ is reached. Hence $I(u)=0$. We have 
  \[
    m_1=S(u)=S(u)-\frac{1}{p+1} I(u)= \left( \frac{1}{2} - \frac{1}{p+1} \right) \|u\|^2_{H^1}\geq m_2.
  \]
  Let $u$ be such that $m_2$ is reached. Then $I(u) \leq 0$. We will prove that $I(u)=0$. By contradiction, we suppose that $ I(u) <0$. As we can see in Figure~\ref{fig:I(ut)-foc} there exists $t_0<1$ such that $I(t_0u)=0$. 
  Therefore we have 
  \[
    \left( \frac{1}{2}- \frac{1}{p+1} \right) \|t_0u\|^2_{H^1} \leq \left( \frac{1}{2}- \frac{1}{p+1} \right) \|u\|^2_{H^1}=m_2,
  \]
  which gives the contradiction. Thus $I(u)=0$. That being the case, we have
  \[
    m_2= \left( \frac{1}{2}- \frac{1}{p+1} \right) \|u\|^2_{H^1}=S(u)\geq m_1.
  \]
  Hence $m_1=m_2$.
  On the other hand from Lemma~\ref{lem:new} of the Fourier rearrangement inequality, we conclude that if $p$ is an odd integer, then there exists $\tilde u \in H^1_{loc}(\R) \cap A_{\frac{T}{2}} $ such that:
  \[
    \tilde u (x) \in \R, \quad \|\tilde u\|_{L^2}=\|u\|_{L^2}, \quad \|\partial_x \tilde u \|_{L^2}=\|\partial_x u\|_{L^2}, \quad \|\tilde u\|_{L^{p+1}} \geq \|u\|_{L^{p+1}}.
  \] 
  Hence the minimizer can be chosen real. Moreover as in the periodic case the minimizer is a solution of~\eqref{eq:ode4} and this concludes the proof.
\end{proof}


\appendix

\section{Triple power nonlinearity}
In this section we treat a special case not covered by the results of the previous sections. Consider the triple power nonlinearity $f(u)=a_1|u|u+a_2|u|^2u+a_3|u|^3u=0$, where $a_1,a_3>0$ and $a_2<0$. We are interested in real valued bounded solutions of~\eqref{eq:ode4}. 

After using the scaling symmetries of~\eqref{eq:nls4}, we may assume $a_1=a_3=1$ and $a_2=-\gamma<0$. Let
\[
  f(\phi)=|\phi|\phi-\gamma|\phi|^2\phi+|\phi|^3\phi.
\]
Denote also
\[
  F(\phi)=\frac{1}{3}|\phi|^3-\frac{\gamma}{4}|\phi|^4+\frac15|\phi|^5.
\]

Changing notation, we set $\omega=-a$, Consider the effective potential
\[
  V(r)=-\omega\frac{r^2}{2}+F(r).
\]
We study the critical points of $V$. Since $f$ is gauge-invariant, $V$ is even in $r$ and we may restrict the study to positive critical points. We have
\[
  V'(r)=-\omega r+f(r).
\]
Define
\[
  f_1(r)=\frac{f(r)}{r}.
\]
The main difference between the present case and the nonlinearities treated in the rest of the paper is that $f_1(r)$ is not strictly increasing, i.e.~\ref{item:h3} is not satisfied.
A positive zero of $V'$ is a positive solution of
\begin{equation}
  0=-\omega+f_1(r)=-\omega+r-\gamma r^2+r^3. \label{eq:2}
\end{equation}
To determine the number of zeros of $V'$, we analyze the variations of $f_1$. We have
\[
  f_1'(r)=1-2\gamma r+3r^2,
\]
which has constant sign when $\gamma<\sqrt{3}$ and otherwise has two (positive) zeros given by
\[
  r_{\pm}=\frac13\left(\gamma\pm\sqrt{\gamma^2-3}\right).
\]
As a consequence, when $0<\gamma\leq \sqrt{3}$, the function $f_1$ is strictly increasing on $[0,\infty)$ and there exists a (unique) positive solution of~\eqref{eq:2} if and only if $\omega>0$.

When $\gamma> \sqrt{3}$, we have $f_1'(r)>0$ for $r\in(0,r_-)\cup(r_+,\infty)$ and $f_1'(r)<0$ for $r\in(r_-,r_+)$. In this case,~\eqref{eq:2} has between $0$ and $3$ solutions. In particular,~\eqref{eq:2} has three positive solutions if and only if $\omega>0$ and
\begin{multline*}
  \frac{1}{27}\left(\gamma(-2\gamma^2 + 9) - 2(\gamma^2 - 3)^{\frac32}\right) =
  f_1(r_+)<\omega<f(r_-)
  =\frac{1}{27}\left(\gamma(-2\gamma^2 + 9) + 2(\gamma^2 - 3)^{\frac32}\right).
\end{multline*}
The $\gamma-\omega$ regions of existence of solutions for~\eqref{eq:2} is represented in the figure below (\textcolor{color1}{zero solution}, \textcolor{color2}{one solution}, \textcolor{color3}{two solutions}, \textcolor{color4}{three solutions}).
\begin{center}
    \includegraphics{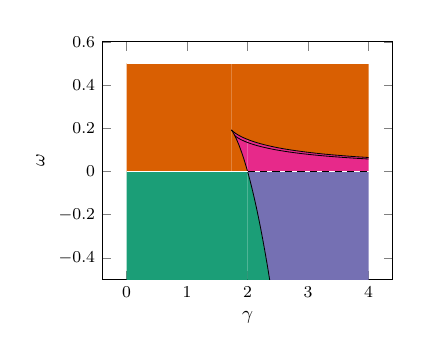}

\end{center}

Whenever they exist, we denote the solutions of~\eqref{eq:2} by
\[
  0<c_1<r_-<c_2<r_+<c_3,
\]
with the convention that when $r_\pm$ do not exist the solution is called $c_1$.

Let us now distinguish the various possibles phase portraits depending on $\gamma$ and $\omega$.

\textbf{Case $\omega<\min\{0,f_1(r_+)\}$}
\label{sec:gamma2-omega0}

In this case the only critical point of $V$ is $0$, which is a center. Solutions of~\eqref{eq:ode4} are all of sn/cn type. The phase portrait is given in Figure~\ref{fig:phase_0-1}.

\begin{figure}[htbp!]
  \centering
  \includegraphics[width=.48\textwidth]{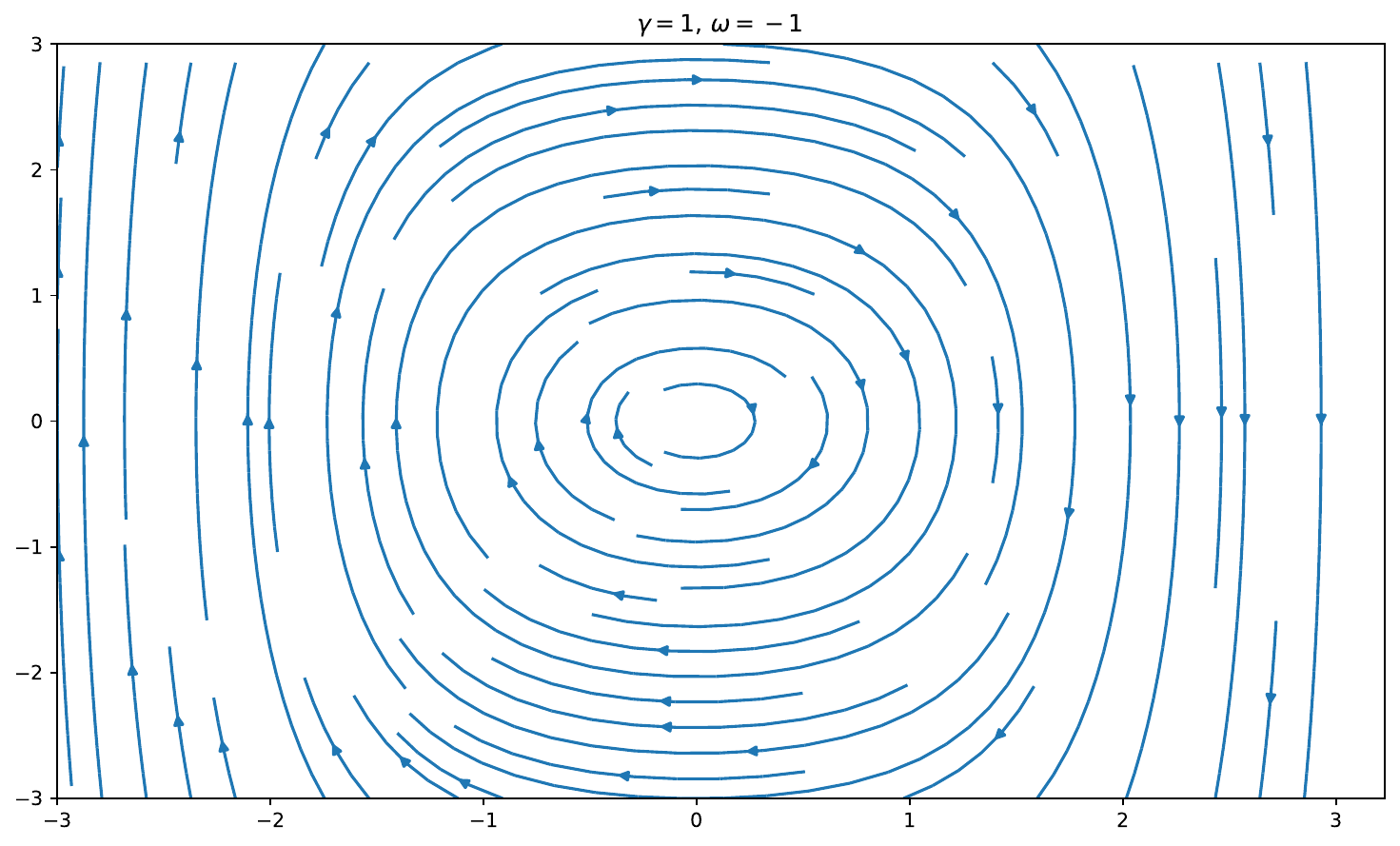}
  \caption{Phase portrait $0$ solution.}
  \label{fig:phase_0-1}
\end{figure}


\textbf{Case $\omega>0$, $\omega\not\in\{(f_1(r_+),f_1(r_-))\}$}
In this case, $V$ has two non-negative critical points: $0$ and $c_1$. The point $0$ is a saddle point. The other critical point $c_1$ is a center. The phase portrait is similar to the one of the single focusing power. We have dn type solutions close to the center and cn type solutions for higher first integrals. The phase portrait is given in Figure~\ref{fig:phase_0-2}.

\begin{figure}[htbp!]
  \centering
  \includegraphics[width=.48\textwidth]{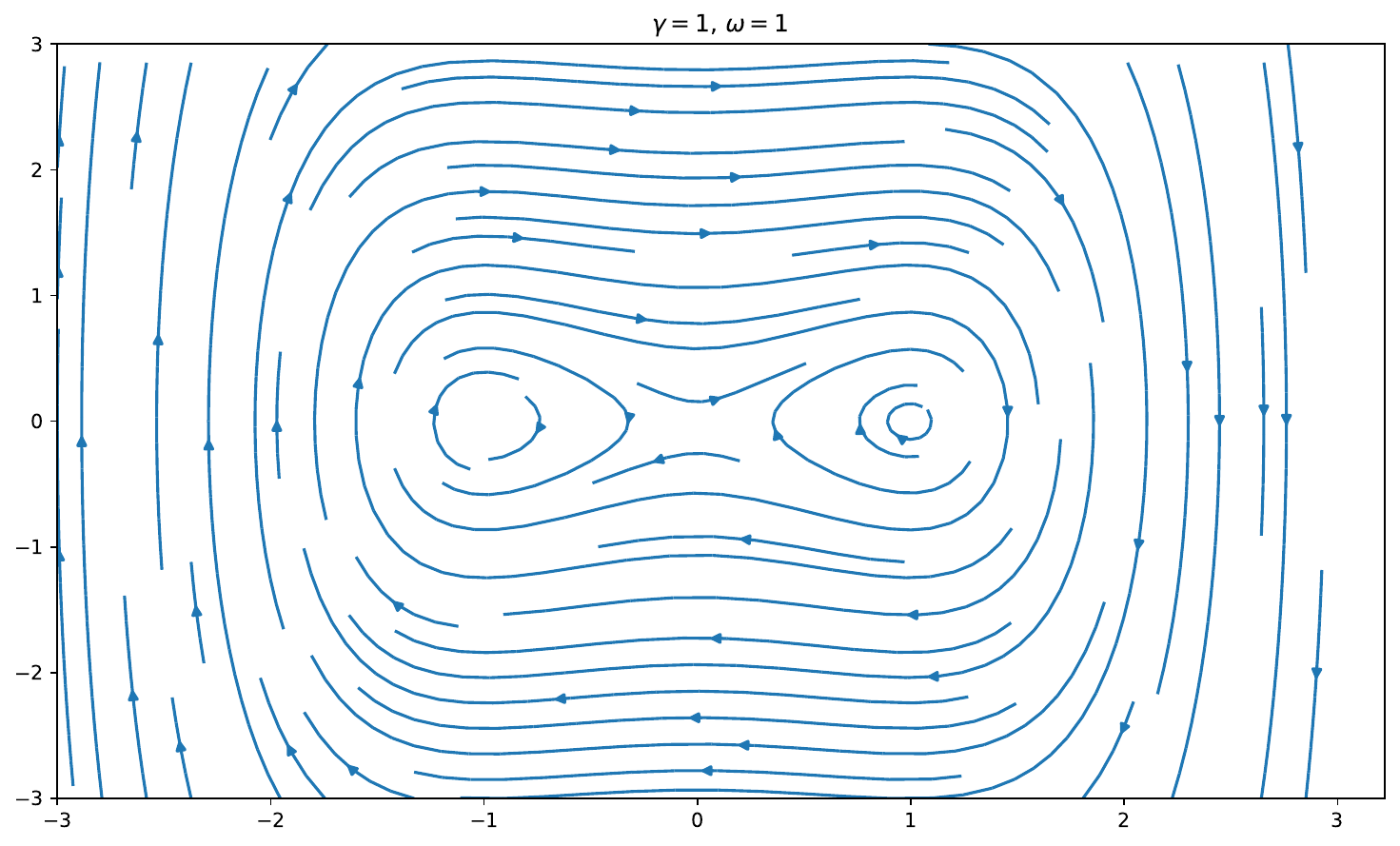}
  \caption{Phase portrait $1$ solution.}
  \label{fig:phase_0-2}
\end{figure}


\textbf{Case $f_1(r_+)<\omega<0$}

In this case, $V$ has three non-negative critical points: $0$ and $c_2,c_3$. The points $0$ and $c_3$ are centers. The other critical point $c_2$ is a saddle point.
The phase portrait is given in Figure~\ref{fig:phase_0-3}.

\begin{figure}[htbp!]
  \centering
  \includegraphics[width=.48\textwidth]{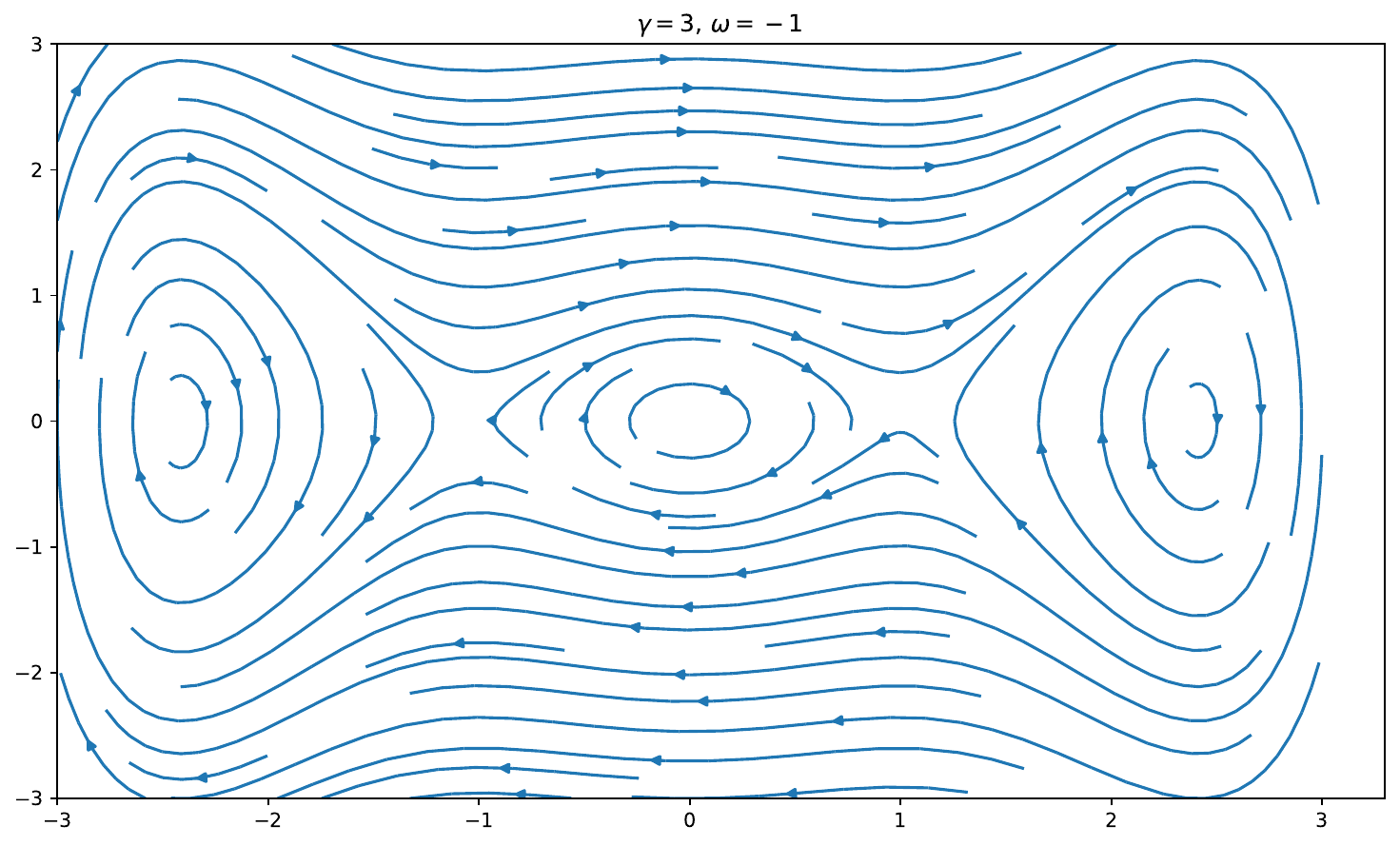}
  \caption{Phase portrait $3$ solution.}
  \label{fig:phase_0-3}
\end{figure}


\textbf{Case $\max(0,f_1(r_+))<\omega<f_1(r_-)$}

In this case, $V$ has four non-negative critical points: $0$ and $c_1,c_2,c_3$. The points $0$ and $c_2$ are saddle points. The other critical point $c_1$ and $c_3$ are centers. There are three possible phase portraits depending on the value of $V(c_2)$. If $V(c_2)>0$, then we have a homoclinic solution connecting $0$ to itself without passing through $c_2$ and an heteroclinic solution connecting $c_2$ to $-c_2$. If $V(c_2)<0$, then the heteroclinic solution connecting $0$ to itself passes through $c_2$ and $c_3$ and there are two homoclinic solutions at $c_2$ (one by lower values and the other by upper values). Finally, at the borderline case $V(c_2)=0$ the main distinguishing feature is a half-kink solution connecting $0$ to $c_2$. In the plane $(\gamma,\omega)$, the half-kink line corresponds to the curve
\[
  \gamma\to
  - \frac{5 \gamma \left(5 \gamma^{2} - 24\right)}{432} + \frac{\sqrt{5} \sqrt{\left(5 \gamma^{2} - 16\right)^{3}}}{432},
\]
starting at the point $\left(\frac{4}{\sqrt{5}},\frac{2\sqrt{5}}{27}\right)$ (observe that this is nothing but the line of non-existence of solitons found in~\cite{LiTsZw21}).
The phase portraits are given in Figure~\ref{fig:phase_0-6}.

\begin{center}
  \includegraphics{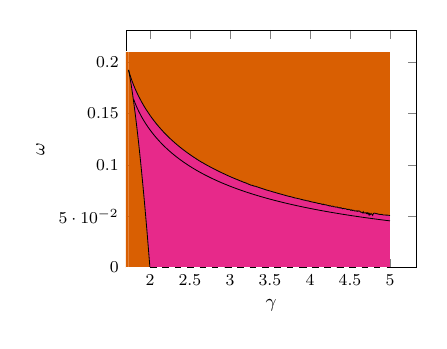}

\end{center}

\begin{figure}[htbp!]
  \centering
  \includegraphics[width=.32\textwidth]{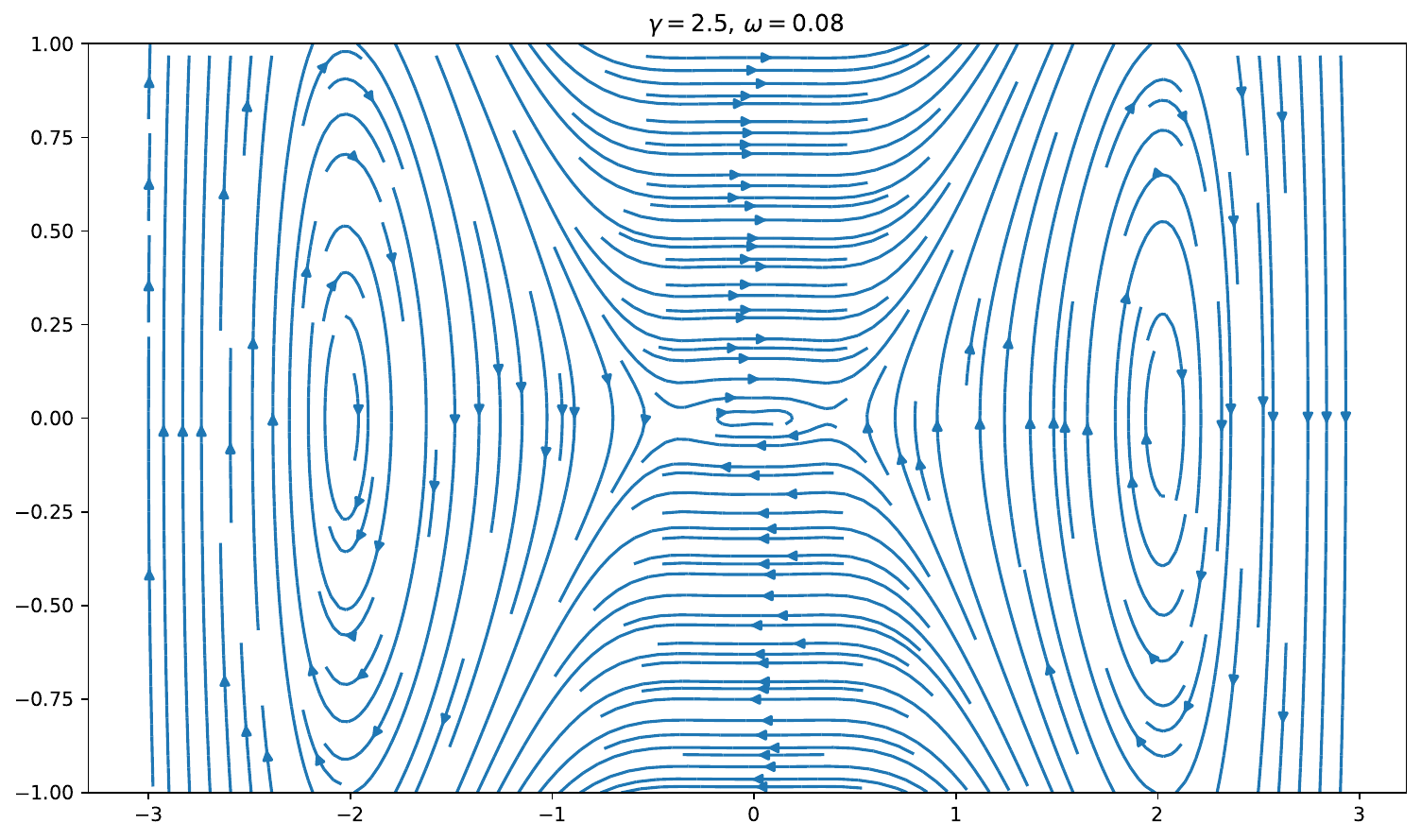}
  \includegraphics[width=.32\textwidth]{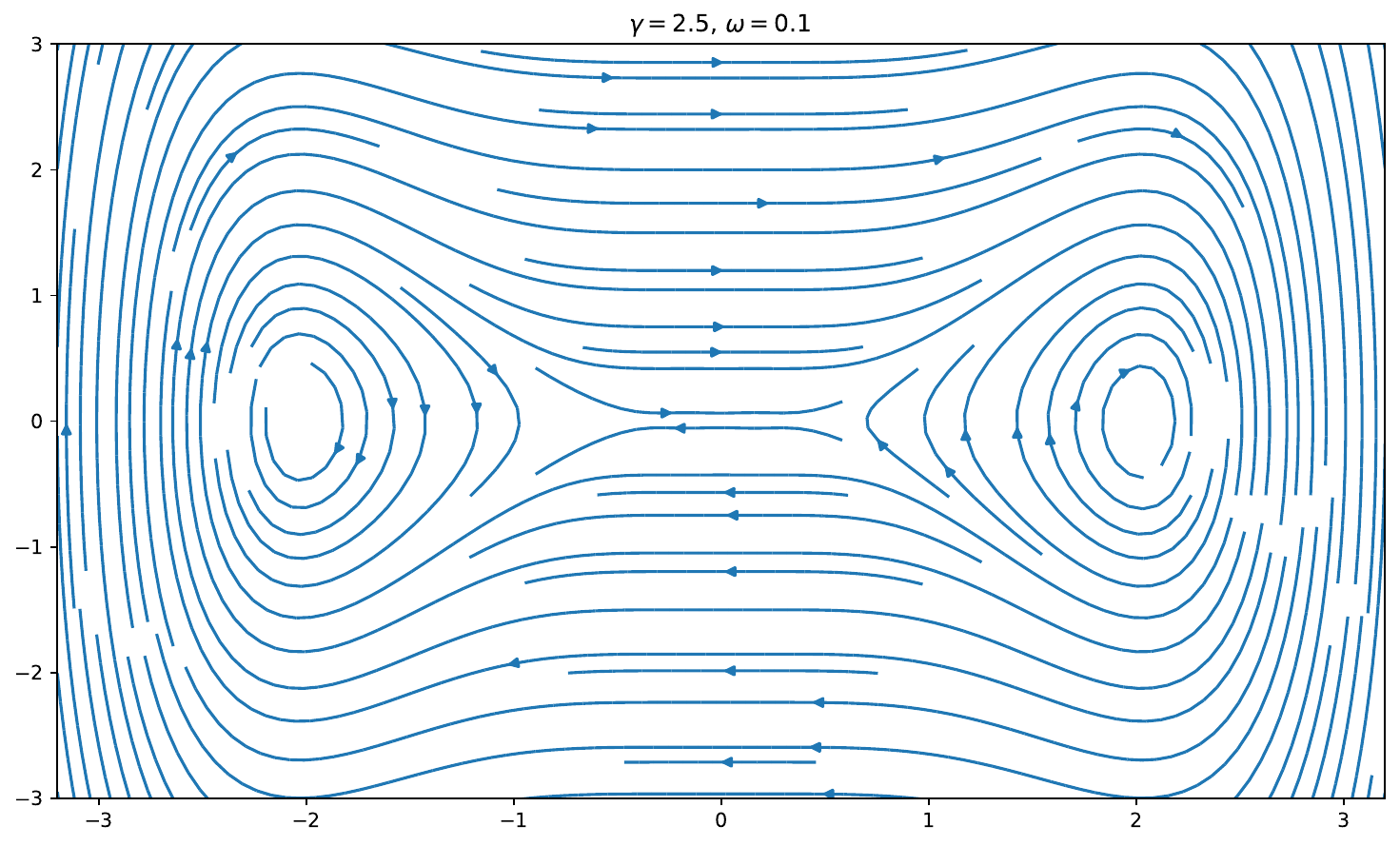}
  \includegraphics[width=.32\textwidth]{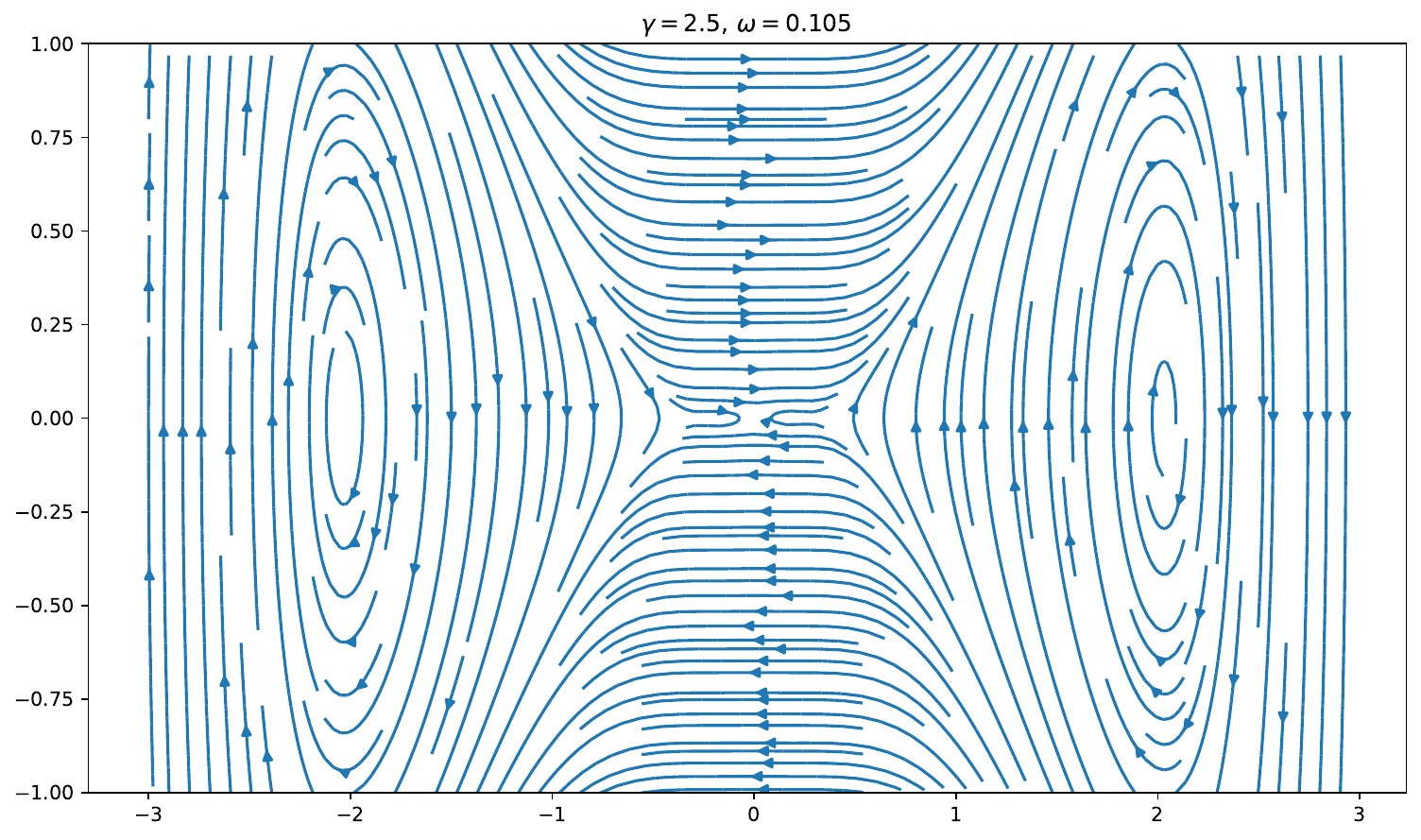}
  \\
  \includegraphics[width=.32\textwidth]{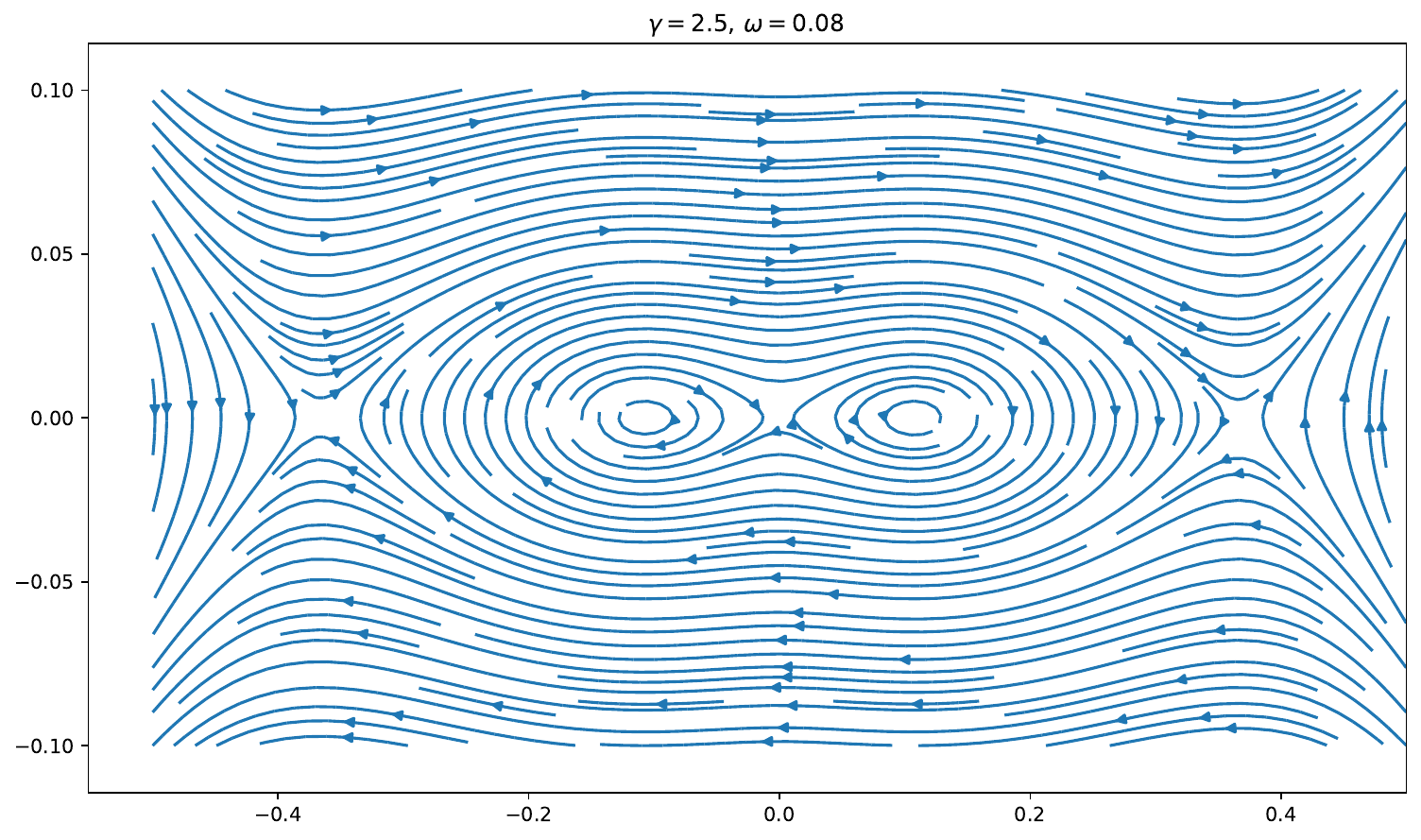}
  \includegraphics[width=.32\textwidth]{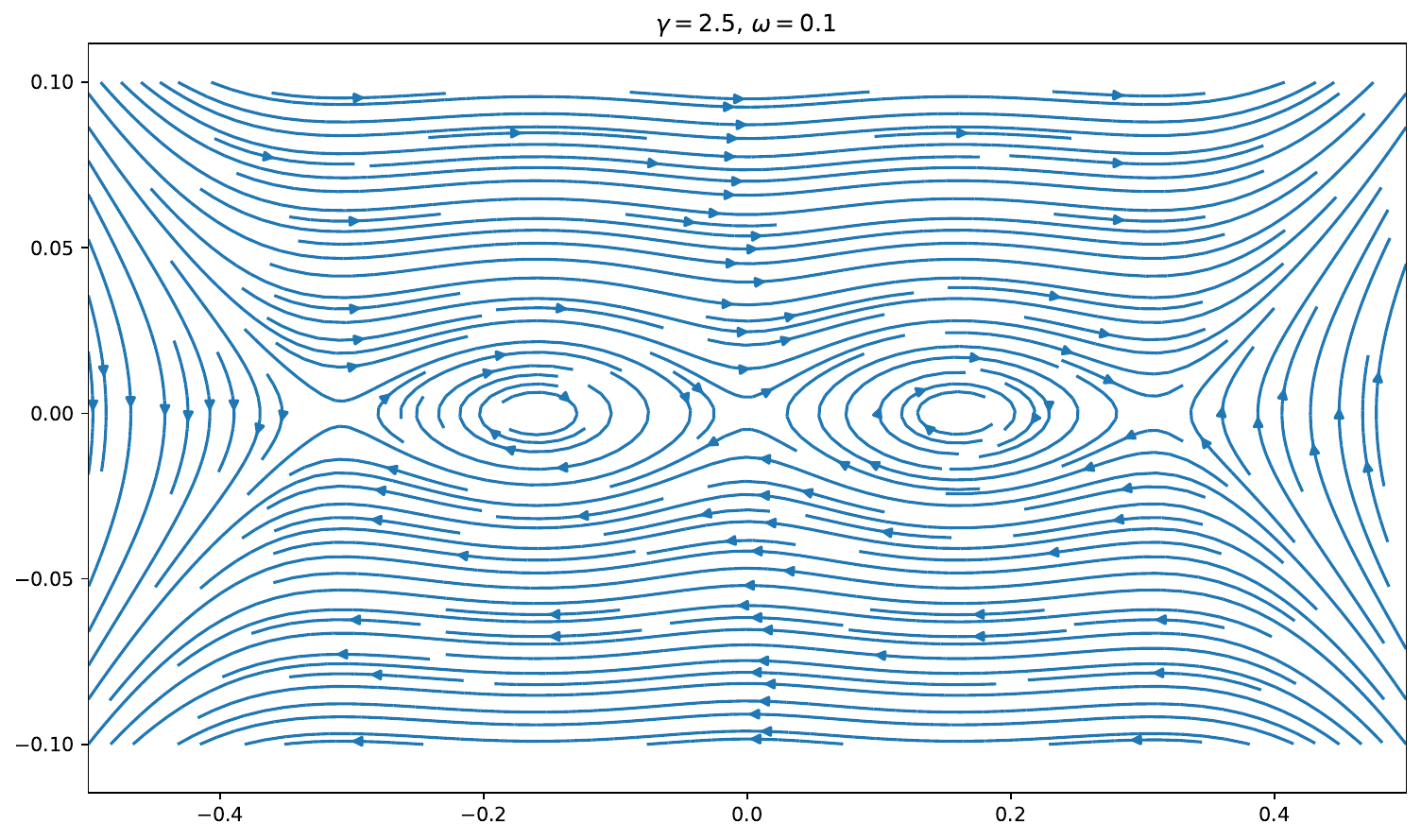}
  \includegraphics[width=.32\textwidth]{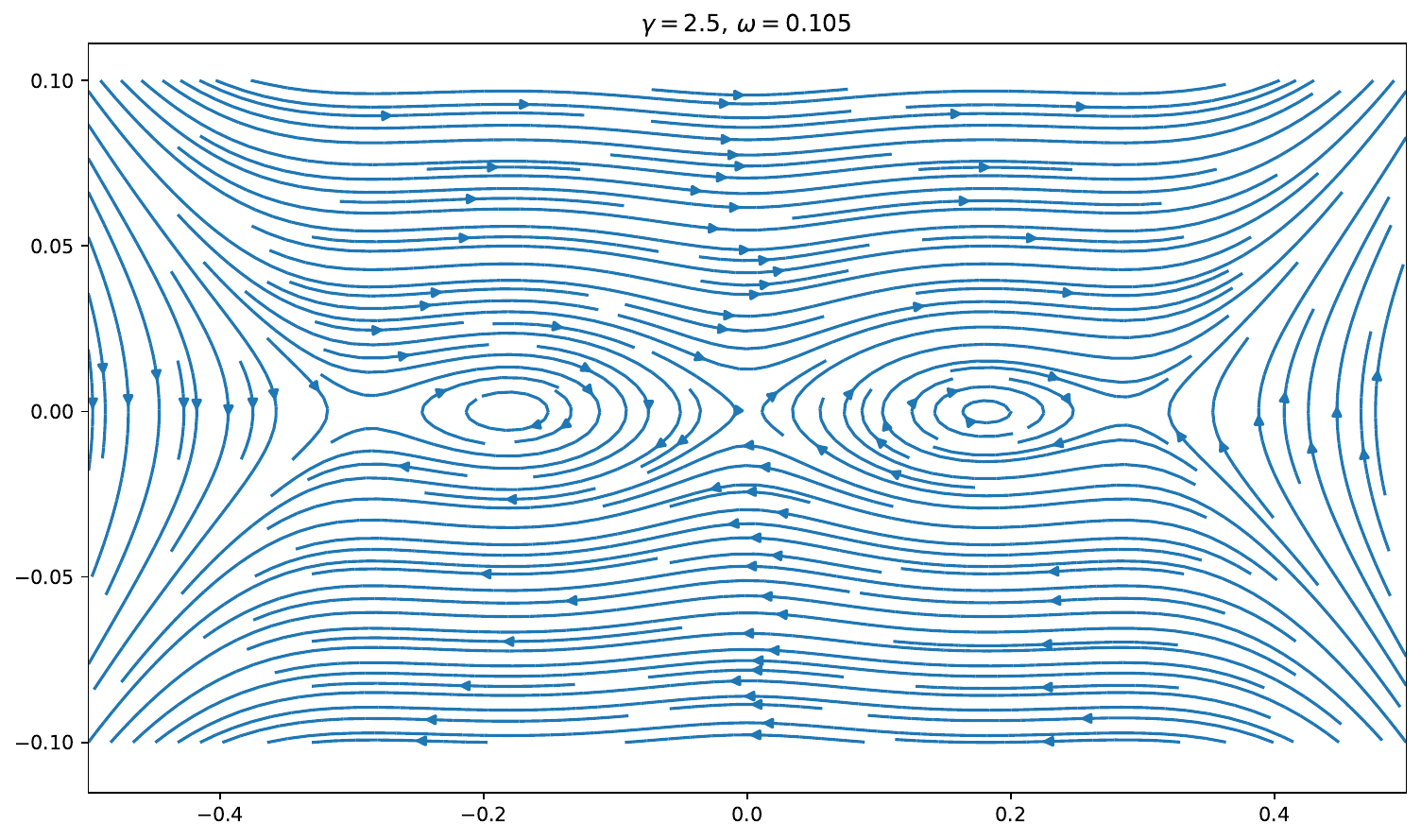}
  \caption{Phase portrait $3$ solution, $V(c_2)>0$, $V(c_2)=0$, $V(c_2)<0$, bottom line is a zoom of the top one.}
  \label{fig:phase_0-6}
\end{figure}

\FloatBarrier

\bibliographystyle{abbrv}
\bibliography{article-bibliographie}

\end{document}